\documentclass[reqno,11pt]{amsart}

\usepackage{amsmath,amssymb,mathtools}
\usepackage[truedimen,margin=25truemm]{geometry}
\usepackage{hyperref}
\usepackage{amsmath,amsfonts,amssymb,graphicx,amsthm,color,yfonts,cite,latexsym}
\usepackage{paralist}
\usepackage{mathrsfs}
\usepackage{hyperref}
\usepackage{physics}
\usepackage{bbm}
\usepackage{extarrows}
\usepackage{bm}
\usepackage{mathtools}
\mathtoolsset{showonlyrefs}
\numberwithin{equation}{section}

\allowdisplaybreaks

\newtheorem{theorem}{Theorem}[section]
\newtheorem{definition}[theorem]{Definition}

\newtheorem{lemma}[theorem]{Lemma}

\newtheorem{remark}[theorem]{Remark}

\makeatletter
\newcommand*{\rom}[1]{\expandafter\@slowromancap\romannumeral #1@}
\makeatother

\newcommand{\Id}{\textup{Id}}

\newcommand{\sgn}{\operatorname{sgn}}

\newcommand{\ls}{\lesssim}

\renewcommand{\lg}{\langle}
\newcommand{\rg}{\rangle}
\newcommand{\R}{\mathbb{R}}

\newcommand{\C}{\mathbb{C}}
\newcommand{\Z}{\mathbb{Z}}

\newcommand{\ep}{\varepsilon}

\newcommand{\be}{\beta}

\newcommand{\CB}{\mathcal{B}}
\newcommand{\CC}{\mathcal{C}}

\newcommand{\CF}{\mathcal{F}}
\newcommand{\CG}{\mathcal{G}}

\newcommand{\CL}{\mathcal{L}}

\newcommand{\CS}{\mathcal{S}}
\newcommand{\CT}{\mathcal{T}}
\newcommand{\CU}{\mathcal{U}}

\newcommand{\CX}{\mathcal{X}}

\newcommand{\SA}{\mathscr{A}}
\newcommand{\SB}{\mathscr{B}}
\newcommand{\SC}{\mathscr{C}}
\newcommand{\SD}{\mathscr{D}}
\newcommand{\SE}{\mathscr{E}}

\newcommand{\FS}{\mathfrak{S}}

\newcommand{\ga}{\gamma}
\newcommand{\de}{\delta}
\newcommand{\ps}{\psi}

\newcommand{\ta}{\tau}

\newcommand{\lm}{\lambda}
\newcommand{\si}{\sigma}
\newcommand{\om}{\omega}

\newcommand{\et}{\eta}
\newcommand{\ze}{\zeta}

\newcommand{\Ph}{\Phi}

\newcommand{\pl}{\partial}
\newcommand{\wt}{\widetilde}
\newcommand{\wh}{\widehat}
\newcommand{\loc}{{\rm loc}}

\newcommand{\Ck}[1]{\left\{#1\right\}}

\newcommand{\Dk}[1]{\left[#1\right]}
\newcommand{\K}[1]{\left(#1\right)}

\newcommand{\No}[1]{\left\| #1 \right\|}

\newcommand{\I}{\infty}

\newcommand{\bx}{\langle x \rangle}

\newcommand{\ft}{{\frac{1}{4}}}
\newcommand{\tw}{\frac{1}{2}}

\newcommand{\ov}{\overline}

\newcommand{\II}{\mathbbm{1}}

\newcommand{\supp}{\operatorname{supp}}

\renewcommand{\Im}{\operatorname{Im}}

\newcommand{\dH}{\dot{H}}
\renewcommand{\pv}{\textup{p.v.}}

\makeatletter
\def\l@section{\@tocline{1}{0pt}{0pt}{}{}}%

\def\l@subsection{\@tocline{2}{0pt}{2.5em}{}{}}%

\def\l@subsubsection{\@tocline{3}{0pt}{3.0em}{}{}}%

\makeatother

\title[Intermediate NLS and Calogero--Moser DNLS]{Small-data $L^2$ theory for the intermediate NLS and the Calogero--Moser derivative NLS}
\author[S. Hadama]{Sonae HADAMA}
\address{The University of Osaka, Japan.}
\email{hadama.sonae.sci@osaka-u.ac.jp}
\date{}

\thanks{The author was supported by JSPS KAKENHI Grant Number 26KJ0203.}

\begin{document}

\begin{abstract}
	In this paper, we study a class of nonlinear Schr\"odinger equations (NLS) in a unified way. This class includes two important examples: the intermediate NLS (INLS) and the Calogero--Moser derivative NLS (CM-DNLS). Our main results are twofold. First, we prove small-data global well-posedness in $L^2(\mathbb{R})$ for a broad class of equations. This includes both focusing and defocusing CM-DNLS and the INLS for arbitrary choices of its parameters. Second, we prove small-data scattering in $L^2(\mathbb{R})$ under an additional assumption. This result covers both focusing and defocusing CM-DNLS and the INLS for specific choices of its parameters. Both the formulation of the problem, including the notion of solution, and the proofs rely crucially on a linear theory for Schr\"odinger equations with rough time-dependent potentials. This theory is also of independent interest, since we allow potentials so rough that the standard Duhamel formulation may not make sense. Our approach is perturbative and is built on the bilinear Strichartz estimate proved by Ozawa and Tsutsumi in 1998. In particular, our argument does not rely on integrability.
\end{abstract}

\maketitle
\tableofcontents
\section{Introduction}\label{sec:intro}

\subsection{Main results}\label{subsec:main II}
In this paper, we study small-data well-posedness and small-data scattering for a class of nonlinear Schr\"odinger equations (NLS) of the form
\begin{equation}\label{eq:CMDNLS}
	i\partial_tu+\partial_x^2u = (p(D)|u|^2) u + q |u|^2 u, \quad u:\R \times\R \to \C,
\end{equation}
where $p(D)$ is a Fourier multiplier with a symbol $p:\R_\xi \to \C$, and $q \in\C$ is a constant.
Throughout this paper, we assume 
\begin{equation}\label{eq:condition a}
	\|p\|_{L_{-1}^\I}:= \operatorname*{ess\,sup}_{\xi\ne0} \abs{\frac{p(\xi)}{\xi}} <\I.
\end{equation}
The equation \eqref{eq:CMDNLS} is a generalization of the intermediate NLS (INLS)
\begin{equation}\label{eq:the INLS}
	i\pl_t u + \pl_x^2 u = \be \Dk{(-i+\CT_h)\pl_x(|u|^2)}u + \ga|u|^2u,
\end{equation}
where $\be,\ga\in\R$, $h\in (0,\I)$, and $\CT_h$ is the Fourier multiplier with symbol $-i\coth(h\xi)$.
Indeed, we can see that \eqref{eq:the INLS} is a special case of \eqref{eq:CMDNLS} by setting
\begin{equation}\label{eq:correspondence}
	p(\xi):=\beta\left(\xi(1+\coth(h\xi))-\frac{1}{h}\right),\qquad q:=\gamma+\frac{\beta}{h}.
\end{equation}
Moreover, if we choose $p(\xi)=\pm(|\xi|+\xi)$ and $q=0$, we can recover the focusing and defocusing Calogero--Moser derivative NLS (CM-DNLS)
\begin{equation}\label{eq:CM}
	i\pl_t u + \pl_x^2 u = \pm (|D| + D)(|u|^2)u
\end{equation}
from \eqref{eq:CMDNLS}. Note that, at least formally, we obtain \eqref{eq:CM} from \eqref{eq:the INLS} by taking $\be=\pm 1$, $\ga=0$, and $h\to \I$.
Since $h$ represents the water depth in the physical model, the limit $h\to\I$ is called the deep-water limit.

Our first main result is the small-data well-posedness in $L^2(\R)$ for \eqref{eq:CMDNLS} with general $p(\xi)$ and $q$.
In particular, Theorem \ref{th:LWP GWP} includes two important cases:
\begin{itemize}
	\item INLS with arbitrary $\be,\ga\in\R$ and $h\in(0,\I)$;
	\item Focusing and defocusing CM-DNLS.
\end{itemize}

\begin{theorem}[Small-data well-posedness for \eqref{eq:CMDNLS}] \label{th:LWP GWP}
	Let $p:\R_\xi \to \C$ satisfy \eqref{eq:condition a}.
	Let $q\in\C$.
	Then, there exist small constants $\ep_0 = \ep_0(\|p\|_{L^\I_{-1}},|q|)>0$, $R=R (\|p\|_{L^\I_{-1}},|q|)>0$, and $T=T(\|p\|_{L^\I_{-1}},|q|)>0$ such that the following statements are true.
	\begin{enumerate}[$(i)$]
		\item (Existence and uniqueness) If $\|\phi\|_{L^2(\R)} \le \ep_0$, there exists a unique solution $u\in C_t(I;L^2_x(\R)) \cap L^4_t(I;L^\I_x(\R))$ to \eqref{eq:CMDNLS} with the initial condition $u(0)=\phi$ such that $$\|u\|_{C_t(I;L^2_x(\R)) \cap L^4_t(I;L^\I_x(\R))}\le R\quad \text{and} \quad \||u|^2\|_{L^2_t(I;\dH^{1/2}_x(\R))} \le R,$$
		where $I=[-T,T]$.
		See Definition \ref{def:solution} for the precise meaning of solution.
		\item (Lipschitz continuity of the flow) Moreover, the data-to-solution map constructed in $(i)$
		\begin{equation}
			\Ck{\ps : \|\ps\|_{L^2_x(\R)} \le \ep_0}\ni \phi \mapsto (u,|u|^2) \in \Dk{C_t(I;L^2_x(\R)) \cap L^4_t(I;L^\I_x(\R))} \times L^2_t(I;\dH^{1/2}_x(\R))
		\end{equation}
		is Lipschitz continuous.
		\item (Global well-posedness) If $p(\xi)$ is real-valued and $q\in \R$, then the local solution obtained in $(i)$ can be extended to a global solution $u\in C_t(\R;L^2_x(\R))$, and its mass is conserved. Namely, for any $t\in \R$, it holds that
		\begin{equation}
			\|u(t)\|_{L^2_x(\R)} = \|\phi\|_{L^2_x(\R)}.
		\end{equation}
		\item (Improved uniqueness) If $p(\xi)$ is real-valued and $q\in \R$, the following stronger uniqueness statement holds.
		Let $\phi\in L^2_x(\R)$ be such that $\|\phi\|_{L^2_x(\R)} \le \ep_0$.
		Let $I'\subset \R$ be an interval such that $0\in I'$.
		Assume that $u'\in C_t(I';L^2_x(\R)) \cap L^4_{\loc,t} (I';L^\I_x)$ is a solution to \eqref{eq:CMDNLS} with the initial condition $u'(0)=\phi$ such that $|u'|^2\in L^2_{\loc,t}(I';\dH^{1/2}_x(\R))$.
		Then, $u(t)=u'(t)$ for all $t\in I'$, where $u$ is the global solution we constructed in $(iii)$. 
	\end{enumerate}
\end{theorem}

\begin{remark}[Definition of homogeneous Sobolev spaces]
	Since several slightly different definitions of homogeneous Sobolev spaces are used in the literature, we explicitly state the definition adopted in this paper.
	For any $\si\in\R$, we define $\dH^\si(\R)$ as follows.
	Let
	\begin{equation}
		\CC := \Ck{f\in \CS(\R): \supp \wh{f} \subset \R\setminus\{0\} }
	\end{equation}
	and $\|f\|_{\dH^\si(\R)} := \||\cdot|^\si \wh{f}\|_{L^2(\R)}$ for all $f\in \CC$.
	Define $\dH^\si(\R)$ as the completion of $\CC$ with respect to the norm $\|\cdot\|_{\dH^\si(\R)}$.
\end{remark}

Our second main result is small-data scattering in $L^2(\R)$ for \eqref{eq:CMDNLS} when $q=0$.
In particular, Theorem \ref{th:CM scattering} includes two important cases:
\begin{itemize}
	\item INLS when $\ga + \be/h  = 0$ (see \eqref{eq:correspondence});
	\item Focusing and defocusing CM-DNLS.
\end{itemize}

\begin{theorem}[Small-data scattering for \eqref{eq:CMDNLS}]\label{th:CM scattering}
	Assume that $p:\R_\xi\to \C$ satisfies \eqref{eq:condition a} and $q=0$.	
	Then, there exist sufficiently small constants $\ep_0=\ep_0(\|p\|_{L^\I_{-1}})>0$ and $R=R(\|p\|_{L^\I_{-1}})>0$ such that the following hold.
	\begin{enumerate}[$(i)$]
		\item (Existence and uniqueness) For any $\phi\in L^2_x(\R)$ such that $\|\phi\|_{L^2_x(\R)} \le \ep_0$, the NLS \eqref{eq:CMDNLS} with the initial condition $u(0)=\phi$ has a unique global solution $u\in C_t(\R;L^2_x(\R))$ such that
		$$\||u|^2\|_{L^2_t(\R;\dH^{1/2}_x(\R))} \le R.$$
		See Definition \ref{def:solution} for the precise meaning of solution.
		\item (Scattering) The solution $u(t)$ obtained in $(i)$ scatters as $t\to\pm \I$. More precisely, there exist $\phi_\pm \in L^2_x(\R)$ such that
		\begin{equation}
			\lim_{t\to\pm \I} \|e^{-it\pl_x^2}u(t)- \phi_\pm \|_{L^2_x(\R)} = 0.
		\end{equation}
	\end{enumerate}
\end{theorem}

\begin{remark}
	In Theorem \ref{th:LWP GWP}, we required the symbol $p(\xi)$ to be real-valued for global well-posedness.
	However, when $q=0$, we obtain small-data global solutions for general complex-valued symbols $p(\xi)$.
\end{remark}

\begin{remark}
	The restriction $q=0$ is not merely technical.
	Already when $p(\xi)\equiv 0$ and $q\in \R\setminus\{0\}$, the cubic NLS on the line is a long-range problem:
	small localized solutions generally exhibit	modified scattering with a logarithmic phase correction rather than	scattering to a free Schr\"odinger solution.
	For this fact, see, for example, \cite{Ozawa 1991,Hayashi Naumkin 1998}.
\end{remark}

\subsection{Comparison with related works}

In this subsection, we compare our results with earlier works on the INLS and the CM-DNLS.
In summary, the discussion below establishes the following points:
\begin{itemize}
\item For the general INLS, our small-data global well-posedness result requires less Sobolev regularity than previous results do;
\item Our small-data scattering result for the CM-DNLS is new and complementary to known results.

\end{itemize}

\subsubsection{The intermediate NLS}
One of the most important examples of the nonlinear equation \eqref{eq:CMDNLS} is the INLS \eqref{eq:the INLS}.
The INLS with $\ga=0$ was introduced by Pelinovsky \cite{Pelinovsky 1995}.
Later, Pelinovsky and Grimshaw derived the INLS with $\ga\ne 0$ in \cite{Pelinovsky Grimshaw 1996}.

To the best of the author's knowledge, de Moura proved the first well-posedness result for the INLS in \cite{de Moura 2007 the INLS}.
More specifically, de Moura proved local well-posedness for small initial data in $H^s(\R)$ with $s\geq1$.
Moreover, small-data global well-posedness in the energy space $H^1(\R)$ was also given.
In \cite{de Moura Pilod 2010}, de Moura and Pilod proved large-data local well-posedness in $H^s(\R)$ for $s>1/2$.
A related result in Besov spaces was obtained in \cite{Barros de Moura Santos 2019}.
In recent preprints \cite{Chapouto et al 2025+,Chapouto et al 2026+}, Chapouto, Forlano, and Laurens lowered the Sobolev threshold for local well-posedness to $s>0$ for all $\beta,\gamma\in\R$ and $0<h\leq\infty$.
When $\ga=0$, they also proved small-data global well-posedness in $H^s(\R)$ with $0<s\le 1$.
To the best of the author's knowledge, this is the best small-data global well-posedness result among the earlier works.

Theorem \ref{th:LWP GWP} $(iii)$ gives small-data global well-posedness in $L^2(\R)$ for the INLS and Calogero--Moser DNLS.  
Indeed, Theorem \ref{th:LWP GWP} $(iii)$ includes the INLS \eqref{eq:the INLS} for all $\be,\ga\in\R$ and $h\in (0,\I)$ by setting $p(\xi)$ and $q$ according to \eqref{eq:correspondence}.
As long as we restrict attention to small initial data, this global well-posedness result should be expected to be optimal.
Indeed, in \cite{Angulo de Moura 2007}, Angulo and de Moura proved that the CM-DNLS, a limiting equation of the INLS, is ill-posed in $H^s(\R)$ with $s<0$, in a very mild sense.
More precisely, they proved that the data-to-solution map cannot be $C^3$ at the origin.

\subsubsection{Calogero--Moser derivative NLS}
Other important examples of \eqref{eq:CMDNLS} are the focusing and defocusing CM-DNLS \eqref{eq:CM}, namely, the cases $p(\xi)=\pm(|\xi|+\xi)$ and $q=0$.
We will review results on these two equations.

In \cite{Gerard Lenzmann 2024}, G\'erard and Lenzmann proved that the focusing CM-DNLS is globally well-posed in the Hardy--Sobolev space $H_+^s(\R)$ for $s\geq1$ under a size assumption on initial data, where $H_+^s(\R)$ is defined as
\begin{equation}
	H_+^s(\R) := \Ck{f \in H^s(\R): \supp \wh{f} \subset [0,\I)}.
\end{equation}
They also classified the traveling waves and studied multisolitons.
In \cite{Killip et al 2025}, Killip, Laurens, and Vi\c{s}an proved global well-posedness for both focusing and defocusing cases in a scaling-critical space $L_+^2(\R):=H^0_+(\R)$ under a size assumption on initial data in the focusing case.

For the focusing CM-DNLS, many results on the long-time behavior, blowup phenomena, and Sobolev norm growth are available.
In \cite{Kim Kwon 2024+}, Kim and Kwon proved soliton resolution for both finite-time blowup solutions and global-in-time solutions, without a symmetry or size assumption.
In particular, they proved that solutions with initial data in $H^{1,1}(\R)$ scatter under a size assumption on the initial data, where 
$$H^{1,1}(\R):= \Ck{u\in H^1(\R): xu \in L^2(\R)}.$$
We should emphasize that they did not use the integrable structure.
Results on the construction and classification of finite-time blowup solutions and on their dynamics were obtained in \cite{Kim Kim Kwon 2024+, Jeong Kim 2024+, Jeong et al 2026+}.
For the growth of Sobolev norms, see \cite{Hogan Kowalski 2024}.

For the defocusing CM-DNLS, in \cite{Chen 2025+}, Chen proved that all solutions with initial data in $L^2_+(\R)$ scatter.
In contrast to the approach in \cite{Kim Kwon 2024+} for the focusing case, the method in \cite{Chen 2025+} makes full use of the integrable structure.

In a recent work \cite{Killip et al 2026+}, Killip, Laurens, Murphy, and Vi\c{s}an obtained pointwise decay estimates and scattering with an explicit convergence rate for focusing and defocusing cases. They required that $\phi\in L^2_+(\R)$, $\|\phi\|_{L^2_x(\R)}$ is small, and
\begin{itemize}
	\item $\phi\in L^1_x(\R)$ for the pointwise decay estimate;
	\item $x\phi \in L^2_x(\R)$ for the scattering with an explicit convergence rate.
\end{itemize}

Finally, we mention two important works in a different direction \cite{Frank Read 2025+, Sun et al 2026+}.
These works seem to be more spectral-theoretic in nature than those mentioned above.

As noted above, the related literature \cite{Killip et al 2025, Kim Kwon 2024+, Chen 2025+, Killip et al 2026+} established the following results:
\begin{itemize}
	\item (Focusing case) Global well-posedness in $L^2_+(\R)$ with a size assumption \cite{Killip et al 2025};
	\item (Focusing case) Scattering when initial data are in $H^{1,1}(\R)$ with a size assumption \cite{Kim Kwon 2024+};
	\item (Defocusing case) Scattering when initial data are in $L^2_+(\R)$ without a size assumption \cite{Chen 2025+};
	\item (Focusing and defocusing cases) Scattering when initial data are in $L^2_+(\R) \cap \bx^{-1} L^2(\R)$ with a size assumption \cite{Killip et al 2026+}.
\end{itemize}
However, Theorem \ref{th:CM scattering} applies to all sufficiently small initial data in $L^2(\R)$, in both the focusing and defocusing cases, without assuming one-sided Fourier support or any additional regularity.
In conclusion, Theorem \ref{th:CM scattering} is new and complementary to known results.

\subsection{Linear Schr\"odinger equations with rough time-dependent potentials}\label{subsec:main I}
In this paper, we rely heavily on a theory of the linear Schr\"odinger equation with a time-dependent potential
\begin{equation}\label{eq:S}
	i\partial_tu+\partial_x^2u-V(t,x)u = 0, \quad u:\R \times\R \to \C,
\end{equation}
where $V(t,x)$ is a \textit{rough} and \textit{complex-valued} potential.
More specifically, throughout this paper, we assume
\begin{equation}\label{eq:condition V}
	V\in L^2_t(I;\dH^{-1/2}_x(\R))
\end{equation}
for an interval $I\subset \R$.
We choose this class of potentials in order to apply the linear analysis to the nonlinear equation \eqref{eq:CMDNLS}.
The connection between the NLS \eqref{eq:CMDNLS} and the linear Schr\"odinger equation \eqref{eq:S} under condition \eqref{eq:condition V} is as follows.
In \cite{Ozawa Tsutsumi 1998}, Ozawa and Tsutsumi proved a bilinear Strichartz estimate
\begin{equation}\label{eq:intro OT}
	\||e^{it\pl_x^2}\phi|^2\|_{L^2_t (\R;\dH^{1/2}_x)} \ls \|\phi\|_{L^2_x(\R)}^2.
\end{equation}
Therefore, it is reasonable to look for solutions $u\in C_t(I;L^2_x(\R))$ such that
\begin{equation}
	\varrho :=	|u|^2 \in L^2_t(I;\dH^{1/2}_x(\R)).
\end{equation}
In \eqref{eq:CMDNLS}, one of the nonlinear terms is given by $(p(D)\varrho)u$, and the assumption \eqref{eq:condition a} for $p(\xi)$  implies
\begin{equation}
	V := p(D)\varrho \in L^2_t(I;\dH^{-1/2}_x(\R)).
\end{equation}
Hence, on the interval $I$, the natural space for potentials is $L^2_t(I;\dH^{-1/2}_x(\R))$. For the scattering problem, we take $I=\R$.

\begin{remark}[Scaling]\label{rmk:scaling}
	Consider the $L^2$ preserving Schrödinger scaling
	$$
	u_\lm(t,x):=\lm^{1/2}u(\lm^2t,\lm x).
	$$
	For the linear Schrödinger equation, the corresponding scaling of
	the potential is
	$$V_\lm(t,x):=\lm^2V(\lm^2t,\lm x).$$
	Since
	$$\|V_\lambda\|_{L_t^2(\R;\dot H_x^{-1/2}(\R))}=\|V\|_{L_t^2(\R;\dot H_x^{-1/2}(\R))},$$
	the potential space considered in this paper is a scaling-critical space.
\end{remark}

When $V$ is a real-valued and sufficiently nice function, we can construct a propagator for this equation \eqref{eq:S}.
Namely, there exists a family of unitary operators $(S_V(t,s))_{t,s\in\R}$ such that
\begin{equation}
	i\pl_t S_V(t,s) = (-\pl_x^2 + V(t,x)) S_V(t,s)
\end{equation}
or, equivalently,
\begin{equation}\label{eq:Duhamel}
	S_V(t,s) = S(t-s) - i\int_s^t S(t-\ta)V(\ta)S_V(\ta,s)d\ta,
\end{equation}
where $S(t):=e^{it\pl_x^2}$.
Moreover, under appropriate short-range assumptions, we can construct inverse wave operators $W_V^{\pm}$ as
\begin{equation}
	W_V^\pm :=\lim_{t\to \pm \I} S(-t)S_V(t,0).
\end{equation}
However, when $V(t,x)\in L^2_t(\R;\dH^{-1/2}_x(\R))$, $V$ need not be an ordinary function.
Hence, the standard Duhamel formula \eqref{eq:Duhamel} does not make sense a priori.
Nevertheless, we can construct a propagator and inverse wave operators for \eqref{eq:S} in an appropriate sense even for such $V$.
\begin{theorem}[Construction of $S_V$ and $W_V^\pm$]\label{th:construction of SV}
	Let $I\subset \R$ be an interval. 
	Assume that $V\in L_t^2(I;\dot H_x^{-1/2}(\R))$.
	Note that $V$ is not necessarily real-valued.
	Then, there exists a propagator for the linear Schr\"odinger equation \eqref{eq:S} in the sense of Definition \ref{def:UV unconditional}, denoted by $(S_V(t,s))_{t,s\in I}$, such that the following hold.
	\begin{enumerate}[$(i)$]
		\item For any $t \in I$, $S_V(t,t)=\Id_{L^2_x(\R)}$.
		\item (Evolution property) For any $t,s,r\in I$, $S_V(t,s)S_V(s,r) = S_V(t,r)$. In particular, $S_V(t,s)^{-1} =S_V(s,t)$.
		\item (Strong continuity) For any $s \in I$, the map
		$$ t \mapsto S_V(t,s)$$
		is strongly continuous in $L^2_x(\R)$. Similarly, for any $t \in I$, the map
		$$ s \mapsto S_V(t,s)$$
		is strongly continuous in $L^2_x(\R)$. 
		\item (Bound of operator norms) There exists a monotone increasing function $\Ph:[0,\I) \to [0,\I)$ such that, for any $t,s\in I$,
		\begin{equation}\label{eq:SV bound}
			\|S_V(t,s)\|_{\CB} \le \Ph\K{\|V\|_{L^2_t(I;\dH^{-1/2}_x)}},
		\end{equation}
		where $\CB$ is the space of bounded linear operators on $L^2_x(\R)$, and $\|\cdot\|_\CB$ is the operator norm.
		\item (Existence of inverse wave operators) When $I=\R$, there exist inverse wave operators $W_V^\pm \in \CB$ such that
		\begin{equation}
			\lim_{t\to\pm \I}\No{S(-t)S_V(t,0) - W_V^\pm }_{\CB} = 0.
		\end{equation}
	\end{enumerate}
\end{theorem}

\begin{remark}
	For the precise definition of $S_V$ in Theorem \ref{th:construction of SV}, see Definitions \ref{def:UV} and \ref{def:UV unconditional}.
	See also Section \ref{subsec:idea}.
\end{remark}

\begin{remark}
	For complex-valued $V$, the propagator is generally not unitary, and the $L^2$ norm need not be conserved.
	The relevant adjoint identity is
	$$
	S_V(t,s)^*=S_{\overline V}(s,t).
	$$
	Since Theorem \ref{th:construction of SV} asserts bounded invertibility of the propagator for general complex-valued $V$, the unitarity of $S_V(t,s)$ is recovered when $V$ is real-valued.
\end{remark}

When $\|V\|_{L^2_t(\R;\dH^{-1/2}_x(\R))}$ is small, the Strichartz estimates hold, which are essential tools for applications to nonlinear problems.
Our second linear result is as follows.
\begin{theorem}[Strichartz estimates for $S_V$]\label{th:Stri}
	There exists a small absolute constant $\ep>0$ such that the following holds.
	For any interval $I\subset \R$ including $0$ and for any $V$ such that $\|V\|_{L^2_t(I;\dH^{-1/2}_x(\R))} \le \ep$, it holds that
	\begin{align}
		&\|S_V(t,0)\phi\|_{C_t(I;L^2_x(\R)) \cap L^4_t(I;L^\I_x(\R))} \ls \|\phi\|_{L^2_x(\R)}, \label{eq:homogeneous Stri}\\
		&\No{\int_0^t S_V(t,\ta)f(\ta) d\ta}_{C_t(I;L^2_x(\R)) \cap L^4_t(I;L^\I_x(\R))} \ls \|f\|_{L^1_t(I;L^2_x(\R)) + L^{4/3}_t(I;L^1_x(\R))},\label{eq:inhomogeneous Stri}\\
		&\No{|S_V(t,0)\phi|^2}_{L^2_t(I;\dH^{1/2}_x)} \ls \|\phi\|_{L^2_x(\R)}^2. \label{eq:bilinear Stri}
	\end{align}
\end{theorem}

\begin{remark}
	See Lemmas \ref{lem:standard Stri}, \ref{lem:difference SV}, and \ref{lem:key} for more detailed versions of Theorem \ref{th:Stri} and difference estimates.
\end{remark}

\begin{remark}
	The linear estimates \eqref{eq:homogeneous Stri} and \eqref{eq:inhomogeneous Stri} are generalizations of the standard Strichartz estimates.
	The bilinear estimate \eqref{eq:bilinear Stri} is a generalization of the bilinear Strichartz estimate \eqref{eq:intro OT} proved by Ozawa and Tsutsumi.
\end{remark}

Finally, we compare Theorems \ref{th:construction of SV} and \ref{th:Stri} with earlier works.

\paragraph{\bf Construction of $S_V$ and Strichartz estimates}
When $V\in L^p_t(I;L^q_x(\R^d))$, the construction of $S_V$ based on the Strichartz estimates for the free flow is available; see, for example, \cite{Yajima 1987, D'Ancona et al 2005}.
	When $V$ is smooth but does not necessarily decay, we can apply the Feynman path integral. In this direction, there are many classical results; see, for example, \cite{Fujiwara 1979, Fujiwara 1980}. 
	See also \cite{Nicola et al 2022 textbook, Fujiwar 2017 textbook} for recent developments.
	
	The Strichartz estimates for $S_V$ often follow as a byproduct of the construction of $S_V$.
	For Strichartz estimates for perturbed flows, see, for example, \cite{Yajima 1987, D'Ancona et al 2005, Rodnianski Schlag 2004}.
	
	When $V$ is not an ordinary function, there are also many results.
	Several authors considered distributional potentials with specific structures.
	See, for example, \cite{Hughes 1995, Hmidi et al 2010, Neidhardt Zagrebonov 2009, Carlone et al 2017, Posilicano 2007}.
	In the author's previous work \cite{Hadama 2026+}, a bounded evolution family was constructed for $V\in L^2_t(I;H_x^{-1/2+\de}(\R))$ without any structural assumptions, where $0<\de\ll 1$.
	We should emphasize that this $\de$ loss was crucial in the construction of $S_V$.
	More specifically, in \cite{Hadama 2026+}, we could use the Christ--Kiselev type lemma thanks to $\de>0$.
	
	The assumption on $V$ in Theorem \ref{th:construction of SV} is essentially different from the assumptions in the earlier works mentioned above.
	First of all, we can include general potentials in $L^2_t(\R;\dH^{-1/2}_x)$ without any structural condition.
	Moreover, we can consider $V$ in a scaling-critical space (see Remark \ref{rmk:scaling}), in which the Christ--Kiselev type lemma is not available, in contrast to the situation in \cite{Hadama 2026+}.

\paragraph{\bf Existence of (inverse) wave operators}
 In Theorem \ref{th:construction of SV}, we defined the inverse wave operators $W_V^\pm$ as $\lim_{t\to\pm\I}S(-t)S_V(t,0)$.
In the literature, the limit $\lim_{t\to\I} S_V(0,t)S(t)$ is often called the wave operator.
	However, in many cases, the existence of the wave operator is equivalent to that of the inverse wave operator.
	
	When $V$ is an ordinary short-range function, there is an extensive literature on the existence of wave operators, but we only refer to \cite{Yafaev 1982, Yoneyama Kato 2020, Mochizuki Motai 2007,Frank et al 2014,Davies 1974,Howland1974}.
	For genuinely long-range potentials, the unmodified wave operators may fail to exist.
	Under suitable assumptions, one instead constructs modified wave operators with an asymptotic phase correction. See, for example, \cite{Kitada 1982,Kitada Yajima 1982}.
	
	Some results are available even if $V$ is not an ordinary function.
	For example, in \cite{Yafaev 1984, Hagedorn 1989}, distributional potentials with some specific structures are considered.
	However, to the best of the author's knowledge, the existence result for inverse wave operators obtained in Theorem \ref{th:construction of SV} is the first result that allows a class of distributional potentials without any structural assumption.

\subsection{Notation}
We collect the basic notation used in this paper.
We use the unitary Fourier-transform conventions
$$\wh{f}(\xi)
:=\frac{1}{\sqrt{2\pi}}\int_{\R}e^{-ix\xi}f(x)dx$$
and
$$\wt{h}(\om,\xi)
:=\frac{1}{2\pi} \iint_{\R^2}e^{-i(t\om+x\xi)}h(t,x)dtdx.$$
We also define the inner product on $L^2(\R)$ by
$$\lg g | f\rg:=\int_{\R}\ov{g(x)}f(x)dx.$$
Let $\CB$ be the space of all bounded linear operators on $L^2(\R)$.

Let $I\subset \R$ be an interval. Define
\begin{equation}
	\CL(I) := \Ck{ f(t,x)\in \CS(\R^2) \mid \wh{f}(t,\xi)\in C^\I_c( I \times (\R\setminus\{0\})) }.
\end{equation}
Then, $\CL(I)$ is dense in $L^2_t(I;\dH^{-1/2}_x(\R))$.

Let $S(t):=e^{it\partial_x^2}$.
For $t,s\in I$ such that $s\le t$, we define a multilinear operator as
\begin{equation}\label{eq:def WV}
	W_{V_1,\dots,V_n}^{(n)}(t,s) = (-i)^n \int_s^t dt_1 \int_s^{t_1}dt_2 \cdots \int_s^{t_{n-1}} dt_n S[V_1](t_1)S[V_2](t_2)\cdots S[V_n](t_n)
\end{equation}
for all $V_1,\dots,V_n \in \CL(I)$, where $S[V](t):=S(-t)V(t)S(t)$.
When $s>t$, we define
\begin{equation}
	W_{V_1,\dots,V_n}^{(n)}(t,s) := W^{(n)}_{\ov{V_n},\dots,\ov{V_1}} (s,t)^*.
\end{equation}
Moreover, define $W_V^{(n)}(t,s) = W_{V,\dots,V}^{(n)}(t,s)$ for $n \ge 1$ and  $W_V^{(0)}(t,s) = \Id_{L^2(\R)}$ for $n=0$.
To shorten the notation, we write $W^{(n)}_{V_1,\dots,V_n}(t):=W^{(n)}_{V_1,\dots,V_n}(t,0)$ and $W^{(n)}_V(t):=W_V^{(n)}(t,0)$.

\section{Formulation of the problem and proof strategy}\label{subsec:idea}
Here we explain the formulation of the problem, including the notion of solution, and the main ideas of the proof. We first describe how the linear theory is applied to the nonlinear problems and then explain how that linear theory is constructed.

\subsection{Nonlinear problems}\label{subsubsec:nonlinear}
Naively, one might try to define a solution to \eqref{eq:CMDNLS} as
\begin{equation}\label{eq:naive Duhamel}
	u(t) = S(t)\phi - i \int_0^t S(t-\ta)\K{p(D)(|u|^2)u + q|u|^2u}(\ta)d\ta. 
\end{equation}
However, it is difficult to give a rigorous meaning to this Duhamel formula \eqref{eq:naive Duhamel} without additional regularity assumptions on $\phi$ because of the derivative loss coming from $p(D)(|u|^2)u$.
Therefore, using the propagator $S_V$ constructed in Theorem \ref{th:construction of SV}, we give the following definition.

\begin{definition}[Definition of a solution to \eqref{eq:CMDNLS}]\label{def:solution}
	Let $I\subset \R$ be an interval and $s\in I$.
	We call $u \in C_t(I;L^2_x(\R)) \cap L^4_{\loc,t}(I;L^\I_x(\R))$ a solution to \eqref{eq:CMDNLS} with initial condition $u(s)=\phi$ if there exists $\varrho \in L^2_{\loc,t}(I;\dH^{1/2}_x(\R))$ such that
	\begin{align}\label{eq:def of solution}
		u(t) = S_{p(D) \varrho}(t,s)\phi - iq \int_s^t S_{p(D)\varrho}(t,\ta)|u(\ta)|^2 u(\ta)d\ta
		\quad \text{and} \quad \varrho(t) = |u(t)|^2.
	\end{align}
\end{definition}
\begin{remark}
	Since $\varrho \in L^2_{\loc,t}(I;\dH^{1/2}_x(\R))$, for any bounded closed interval $J\subset I$, we have $p(D)\varrho \in L^2_t(J;\dH^{-1/2}_x(\R))$.
	Therefore, by Theorem \ref{th:construction of SV}, the propagator $(S_{p(D)\varrho}(t,s))_{t,s\in I}$ is well-defined.
	Moreover, by Theorem \ref{th:construction of SV} $(iv)$, the RHS of \eqref{eq:def of solution} makes sense in $C_t(I;L^2_x(\R))$.
\end{remark}

In the proof of Theorem \ref{th:LWP GWP}, we seek a fixed point $(u,\varrho)$ of the map $F[u,\varrho]=(F_1[u,\varrho], F_2[
u,\varrho])$, where
\begin{align}
	&F_1[u,\varrho](t)
	:=S_{p(D)\varrho}(t,0)\phi
	-iq\int_0^tS_{p(D)\varrho}(t,\tau)|u(\tau)|^2u(\tau)d\tau,\\
	&F_2[u,\varrho]:=|F_1[u,\varrho]|^2.
\end{align}
By the Strichartz estimates for $S_V$ (Theorem \ref{th:Stri}), we can close the estimates for $(u,\varrho)$ on a small ball in $$\Dk{C_t(I;L^2_x)\cap L^4_t(I;L^\I_x)} \times L^2_t(I;\dH^{1/2}_x), \quad I=[-T,T].$$
If $p(\xi)$ is real-valued and $q\in \R$, the $L^2$ conservation law is justified by smooth approximation of the potential.
Hence, the local construction can be iterated to obtain the global solution.

When $q=0$, the problem is reduced to an equation for the density:
\begin{equation}\label{eq:density equation}
	\varrho=F_3[\varrho]:=|S_{p(D)\varrho}(t,0)\phi|^2.
\end{equation}
This reduction to the density $\varrho$ has been used by several authors in the context of the Hartree equation for infinitely many particles. See, for example, \cite{Chen et al 2018, Lewin Sabin 2014}.
By the Ozawa--Tsutsumi type bilinear estimate in Theorem \ref{th:Stri}, we can find a fixed point $\varrho$ of $F_3$.
Define $u(t):=S_{p(D)\varrho}(t,0)\phi$. Then, $u$ is a global solution.
Finally, since $p(D)\varrho\in L^2_t(\R;\dot H^{-1/2}_x)$, Theorem \ref{th:construction of SV} $(v)$ shows that $u$ scatters.

\subsection{Linear problems}\label{subsec:linear}
As explained in Section \ref{subsubsec:nonlinear}, the linear theory (Theorems \ref{th:construction of SV} and \ref{th:Stri}) is crucial in the nonlinear analysis.
Now, we explain how to define and construct $S_V$.

Let $S[V](t)=S(-t)V(t)S(t)$.
A formal iteration of the Duhamel formula gives the Dyson series expansion
\begin{equation}\label{eq:Dyson}
S_V(t,s)=S(t)\sum_{n\ge0}W_V^{(n)}(t,s)S(-s).
\end{equation}
In the sequel, we assume $t>s$ for simplicity.
When $V$ is sufficiently nice, for example $\|V\|_{L^1_t(I;L^\I_x)} \ll 1$, the expansion \eqref{eq:Dyson} is a consequence of the Duhamel formula.
By contrast, in the present paper, we define $S_V(t,s)$ for small $V \in L^2_t (I;\dH^{-1/2}_x)$ by \eqref{eq:Dyson}.
The main technical result is
\begin{equation}\label{eq:idea-dyson-bound}
	\No{W^{(n)}_{V_1,\ldots,V_n}(t,s)}_{\CB}
	\leq \frac{\ep A^n}{n^2}
	\prod_{j=1}^n\|V_j\|_{L^2_t(I;\dH^{-1/2}_x)},
\end{equation}
where $A\ge 1$ and $0<\ep \ll 1$ are absolute constants, and $I=[\min(t,s),\max(t,s)]$.
See Theorem \ref{th:multilinear 1} for the precise statement.

Note that, once the multilinear estimate \eqref{eq:idea-dyson-bound} is established, we can define $S_V(t,s)$ for $V$ such that $\|V\|_{L^2_t(\R;\dH^{-1/2}_x)} <1/A$ via \eqref{eq:Dyson}.
When $V$ is not small, we define
$$S_V(t,s):=S_V(t,r_{N-1}) S_V(r_{N-1},r_{N-2}) \cdots S_V(r_1,s),$$
where $V$ is small on each subinterval of $[s,t]$.
See Definition \ref{def:UV unconditional} for more details.

Once the definition of $S_V$ is established with \eqref{eq:idea-dyson-bound}, we can prove elementary properties of $S_V$ and the existence of $W_V^\pm$ (Theorem \ref{th:construction of SV}) in a standard way.
The linear and bilinear Strichartz estimates (Theorem \ref{th:Stri}) are also consequences of the multilinear estimate \eqref{eq:idea-dyson-bound}.

The main difficulty in proving \eqref{eq:idea-dyson-bound} is the time ordering $t>t_1>\cdots>t_n>s$ in \eqref{eq:def WV}.
Indeed, if we remove this ordering, we find
\begin{equation}
	\No{\int_I \cdots \int_I S[V_1](t_1)\cdots S[V_n](t_n) dt_1\cdots dt_n}_\CB \le \prod_{j=1}^n \No{\int_I S[V_j](t_j)dt_j}_\CB.
\end{equation}
Hence, we can reduce the problem to the estimate for the first order term, and it is already known (see, for example, \cite[Lemma 2.1]{Hadama 2026+}):
\begin{equation}
	\No{\int_I S[V](t) dt}_{\CB} \le C_1 \|V\|_{L^2_t(I;\dH^{-1/2}_x)}
\end{equation}
for an absolute constant $C_1>0$.
Therefore, we conclude that
\begin{equation}\label{eq:non-retarded}
	\begin{aligned}
	\No{\int_I \cdots \int_I S[V_1](t_1)\cdots S[V_n](t_n) dt_1\cdots dt_n}_\CB
	&\le C_1^n \prod_{j=1}^n \|V_j\|_{L^2_t(I;\dH^{-1/2}_x)} \\
	&\le \frac{\ep A^n }{n^2} \prod_{j=1}^n\|V_j\|_{L^2_t(I;\dH^{-1/2}_x)}
	\end{aligned}
\end{equation}
by choosing appropriate $A$ and $\ep$.
One may try to use the Christ--Kiselev type lemma (see \cite{Christ Kiselev 2001}; see also \cite[Lemma A.2]{Hadama 2026+}) to obtain \eqref{eq:idea-dyson-bound} directly from \eqref{eq:non-retarded}, but it is impossible.  

In this paper, we first prove \eqref{eq:idea-dyson-bound} for $n=1,2$ (Section \ref{subsec:1st and 2nd}) and extend it to general $n \ge 3$ by induction (Section \ref{subsec:higher order}).
Since the case $n=1$ is already known, the most important part is the proof of the case $n=2$. 
For the second order term, we use
$$
\II_{\{t_1>t_2\}}
=\tw\bigl(1+\operatorname{sgn}(t_1-t_2)\bigr)
\quad \text{and} \quad
\sgn(t_1-t_2)= \frac{1}{i\pi}\pv \int_{-\I}^\I
\frac{e^{i(t_1-t_2)\lm}}{\lm} d\lm.
$$
The constant part factorizes into two first-order operators.
By applying a principal-value representation of the sign function, taking Fourier transforms, and changing variables, we reduce the remaining part to a singular-integral estimate associated with the dilation group
$$
(U_sG)(u,v,w)=G(u,e^sv,e^{-s}w).
$$
Logarithmic coordinates turn this dilation into translation, and the resulting one-dimensional multiplier is uniformly bounded.
This proves \eqref{eq:idea-dyson-bound} for $n=2$.

\section{Key multilinear estimate}\label{sec:multilinear estimate}
In this section, we prove the following multilinear estimate, which is the essential tool in this article.
\begin{theorem}[Key estimate]\label{th:multilinear 1}
	There are absolute constants $0< \ep\ll1$ and $A\ge 1$ such that, for any $n\in\Z_{\ge 1}$, for any $t,s\in \R$, and for any $V_1,\dots,V_n\in\CL(I)$, the multilinear estimate 
	\begin{equation}\label{eq:strong induction 1}
		\No{W^{(n)}_{V_1,\dots,V_n}(t,s)}_{\CB} \le \frac{\ep A^n}{n^2} \prod_{j=1}^n \|V_j\|_{L^2_t(I;\dH^{-1/2}_x(\R))}
	\end{equation}
	holds, where $I=[\min(t,s),\max(t,s)]$.
\end{theorem}
\begin{remark}
	By Theorem \ref{th:multilinear 1}, we can define $W^{(n)}_{V_1,\dots,V_n}(t,s)$ for general $V_1,\dots,V_n\in L^2_t(I;\dH^{-1/2}_x(\R))$
	as a unique extension because $\CL(I)$ is dense in $L^2_t(I;\dH^{-1/2}_x(\R))$.
\end{remark}

\subsection{First and second order terms}\label{subsec:1st and 2nd}
The first term can be estimated easily.
\begin{lemma}[First order term]\label{lem:first order term}
	There exists an absolute constant $C_1>0$ such that, for any $t,s\in\R$ and $V_1\in \CL(I)$, the multilinear estimate
	\begin{equation}
		\No{W_{V_1}^{{(1)}}(t,s)}_{\CB} \le C_1 \|V_1\|_{L^2_t (I;\dH^{-1/2}_x(\R))}
	\end{equation}
	holds, where $I=[\min(t,s),\max(t,s)]$.
\end{lemma}

\begin{proof}
	A proof is given in \cite{Hadama 2026+}, but we give a proof here for completeness.
	Let $\FS^2$ be the Hilbert--Schmidt class on $L^2(\R)$ and recall
	$$\|A\|_\CB\le \|A\|_{\FS^2} = \|A(x,x')\|_{L^2_{x,x'}(\R^2)}.$$
	Then, by unitarity of the Fourier transform, we obtain
	\begin{align}
		\No{W_{V_1}^{{(1)}}(t,s)}_{\CB}
		&\le  \No{W_{V_1}^{{(1)}}(t,s)}_{\FS^2}
		= \No{\CF \int_s^t S(-t)V_1(t)S(t)dt \CF^{-1}}_{\FS^2} \\
		&=\frac{1}{\sqrt{2\pi}} \No{\int_s^t e^{it(\xi^2-\et^2)}\wh{V_1}(t,\xi-\et)dt}_{L^2_{\xi,\et}(\R^2)} 
		= \No{\wt{V_1 \II_I}(-\xi^2+\et^2,\xi-\et)}_{L^2_{\xi,\et}(\R^2)} \\
		&\ls \|V_1\|_{L^2_t(I;\dH^{-1/2}_x)},
	\end{align}
	where $\wt{V_1\II_I}(\xi,\et)$ is the space-time Fourier transform of $V_1(t,x)\II_I(t)$.
\end{proof}

In the sequel, we consider the second order term.	
\begin{lemma}[Second order term]\label{lem:second order term}
	There exists an absolute constant $C_2>0$ such that,
	for any $t,s\in\R$ and for any $V_1,V_2\in \CL(\R)$, the multilinear estimate
	\begin{equation}
		\No{W_{V_1,V_2}^{{(2)}}(t,s)}_{\CB(L^2_x(\R))} \le C_2 \prod_{j=1,2}\|V_j\|_{L^2_t (I;\dH^{-1/2}_x(\R))}
	\end{equation}
	holds, where $I=[\min(t,s),\max(t,s)]$.
\end{lemma}

To prove Lemma \ref{lem:second order term}, we use the following lemma, which will be proved later.
\begin{lemma}\label{lem:multi linear}
	For any $f,g\in L^2(\R)$ and for any $h_1,h_2\in L^2(\R^2)$,
	\begin{multline}\label{eq:4th linear}
		\limsup_{\ep \to 0}\abs{\int_{\ep\le |\lm|\le 1/\ep} \frac{d\lm}{\lm} \iiint_{\R^3} dpdqdy
			\ov{g(y+p)}f(y-q)h_1\K{p,y+\frac{p+q\lm}{2}}
			h_2\K{q,y-\frac{q+\lm p}{2}}}\\
		\ls \|h_1\|_{L^2(\R^2)}\|h_2\|_{L^2(\R^2)}\|f\|_{L^2(\R)}\|g\|_{L^2(\R)}.
	\end{multline}
\end{lemma}

We first prove Lemma \ref{lem:second order term} assuming Lemma \ref{lem:multi linear}.
\begin{proof}[Proof of Lemma \ref{lem:second order term} assuming Lemma \ref{lem:multi linear}]
Since
$$\No{W_{V_1,V_2}^{(2)}(t,s)}_\CB = \No{W_{\ov{V_2},\ov{V_1}}^{(2)}(s,t)}_\CB,$$
we can assume $t>s$. The elementary identity
\[
\II_{\{t_1>t_2\}}=\tw \K{1+\sgn(t_1-t_2)}
\]
gives 
\begin{align}
		W_{V_1,V_2}^{(2)}(t,s)
		&=-\tw \int_I S[V_1](t_1) dt_1 \int_I S[V_2](t_2)dt_2
		-\tw \iint_{I^2}\sgn(t_1-t_2)
		S[V_1](t_1) S[V_2](t_2) dt_2 dt_1 \\
		&=: \SA + \SB.
\end{align}
Hence, it follows that
\begin{equation}
	\No{W_{V_1,V_2}^{(2)}(t,s)}_{\CB}
	\le \|\SA\|_{\CB} + \|\SB\|_{\CB}.
\end{equation}
For $\SA$, by Lemma \ref{lem:first order term}, 
\begin{equation}
\|\SA\|_{\CB} \le \tw \No{W_{V_1}^{(1)}(t,s)}_{\CB} \No{W_{V_2}^{(1)}(t,s)}_{\CB}
\ls \prod_{j=1,2} \|V_j\|_{L^2_t(I;\dH^{-1/2}_x)}.
\end{equation}
Finally, we estimate the second term $\SB$.
Let $f, g \in \CS(\R)$.
Then, we have
\begin{align}
	&\int_\R \ov{g(x)} \Big(S[V_1](t_1) S[V_2](t_2)f\Big)(x) dx \\
	&\quad =\frac{1}{2\pi}\iiint_{\R^3} \ov{\wh{g}(\xi_1)} \wh{f}(\xi_3) \wh{V_1}(t_1,\xi_1-\xi_2) \wh{V_2}(t_2,\xi_2-\xi_3) e^{it_1(|\xi_1|^2 -|\xi_2|^2)} e^{it_2(|\xi_2|^2 - |\xi_3|^2)}d\xi_1 d\xi_2 d\xi_3 
	\\
	&\quad =\frac{1}{2\pi}\iiint_{\R^3}
	\ov{\wh{g}(y+p)}\wh{f}(y-q)
	\widehat V_1(t_1,p)\widehat V_2(t_2,q)
	e^{2it_1p(y+p/2)}e^{2it_2q(y-q/2)} dpdqdy,
\end{align}
where we made the change of variables $p:=\xi_1-\xi_2$, $q:=\xi_2-\xi_3$ and $y=\xi_2$. 
Since
\begin{align}\label{eq:sign}
	\sgn(t_1-t_2)
	= \frac{1}{i\pi}\pv \int_{-\I}^\I
	\frac{e^{i(t_1-t_2)\lm}}{\lm} d\lm
	= \frac{\sgn(pq)}{i\pi}\pv \int_{-\I}^\I
	\frac{e^{ipq(t_1-t_2)\lm}}{\lm} d\lm
\end{align}
for any $p,q\in\R \setminus\{0\}$, we can compute
\begin{align}
	\lg g | \SB f\rg
	&=-\tw \int_\R \ov{g(x)} \bigg[\iint_{I^2} \sgn(t_1-t_2)S[V_1](t_1) S[V_2](t_2) dt_2dt_1 \bigg]f(x) dx \\
	&=\frac{i}{4\pi^2}\pv \int_{-\I}^\I \frac{d\lm}{\lm} \int_I dt_1 \int_I dt_2 \iiint_{\R^3} dpdqdy \\
	&\qquad \cdot \ov{\wh{g}(y+p)}\wh{f}(y-q)
	\wh{V_1}(t_1,p)\wh{V_2}(t_2,q)
	\sgn(pq)e^{ipq(t_1-t_2)\lm }e^{2it_1p(y+p/2)}e^{2it_2q(y-q/2)} \\
	&= \frac{i}{2\pi}\pv \int_{-\I}^\I \frac{d\lm}{\lm} \iiint_{\R^3} dpdqdy  \ov{\wh{g}(y+p)}\wh{f}(y-q)\\
	&\qquad \cdot
	\sgn(p)\wt{V_1\II_I}\K{-2p\K{y+\frac{p+q\lm}{2}},p}
	\sgn(q)\wt{V_2\II_I}\K{-2q\K{y-\frac{q+p\lm}{2}},q},
\end{align}
where $\wt{h\II_I}(\om,\xi)$ is the space-time Fourier transform of $h(t,x)\II_I(t)$. 
By Lemma \ref{lem:multi linear}, we conclude that
\begin{align}
	|\lg g | \SB f\rg|
	&\ls \No{\wt{V_1 \II_I}(\om\xi,\xi)}_{L^2_{\om,\xi}} \No{\wt{V_2\II_I}(\om\xi,\xi)}_{L^2_{\om,\xi}} \|f\|_{L^2_x} \|g\|_{L^2_x} \\
	&= \|V_1\|_{L^2_t(I;\dH^{-1/2}_x)} \|V_2\|_{L^2_t(I;\dH_x^{-1/2})} \|f\|_{L^2_x} \|g\|_{L^2_x}.
\end{align}
Therefore, 
\begin{equation}
	\|\SB\|_{\CB} \ls \|V_1\|_{L^2_t(I;\dH^{-1/2}_x)} \|V_2\|_{L^2_t(I;\dH_x^{-1/2})},
\end{equation}
which completes the proof.
\end{proof}

Finally, we prove Lemma \ref{lem:multi linear}.
\begin{proof}
\noindent \textbf{Step 0: Setup of the proof.}
For a very technical reason, we will show
\begin{multline}\label{eq:4th linear 2}
	\limsup_{\ep\to0}\abs{\int_{\ep\le|\lm|\le1/\ep} \frac{d\lm}{\lm} \iiint_{\R^3} dpdqdy
		\ov{g(y+p)}f(y-q)h_1\K{p,y+\frac{p}{2} + \lm q}
		h_2\K{q,y-\frac{q}{2} -\lm p }}\\
	\ls \|h_1\|_{L^2(\R^2)}\|h_2\|_{L^2(\R^2)}\|f\|_{L^2(\R)}\|g\|_{L^2(\R)}
\end{multline}
instead of \eqref{eq:4th linear}.
Making the change of variables in $\lm$, we can see that \eqref{eq:4th linear} and \eqref{eq:4th linear 2} are equivalent.
For any $\lm \ne 0$, define
	\begin{equation}
		I(\lm)
		:= \iiint_{\R^3} dpdqdy
		\ov{g(y+p)}f(y-q)h_1\K{p,y+\frac{p}{2} + \lm q}
		h_2\K{q,y-\frac{q}{2} -\lm p }.
	\end{equation}
	Then, we have
	\begin{align}
	&\abs{\int_{\ep\le|\lm|\le1/\ep} \frac{d\lm}{\lm} \iiint_{\R^3} dpdqdy
		\ov{g(y+p)}f(y-q)
		h_1\K{p,y+\frac{p}{2} + \lm q}	h_2\K{q,y-\frac{q}{2} -\lm p }}\\
	&\quad \ls \int_{1/2\le |\lm|\le 1/\ep} \abs{\frac{I(\lm)}{\lm}} d\lm 
	+\abs{\int_{\ep\le|\lm|\le 1/2} \frac{I(\lm)}{\lm} d\lm} =:\SA + |\SB|.
	\end{align}
	
	\noindent \textbf{Step 1: Estimate of $\SA$.}
We have
	\begin{align}
		|I(\lm)|
		&= \frac{1}{\lm^2}\bigg|\iiint_{\R^3} dpdqdy \\
		&\quad \cdot \ov{g\K{\frac{p}{\lm}}}f\K{\frac{q}{\lm}}
		              h_1\K{-y+\frac{p}{\lm}, - q+\frac{p}{2\lm} + \frac{(1+2\lm)y}{2}}
			          h_2\K{y-\frac{q}{\lm},-p+\frac{q}{2\lm}+\frac{(1+2\lm)y}{2}}\bigg| \\
		&\le \frac{1}{\lm^2} \No{f\K{\frac{q}{\lm}}
			                 h_2\K{y-\frac{q}{\lm},-p+\frac{q}{2\lm}+\frac{(1+2\lm)y}{2}}}_{L^2_{p,q,y}} \\
		&\qquad  \cdot \No{g\K{\frac{p}{\lm}} h_1\K{-y+\frac{p}{\lm},-q+\frac{p}{2\lm}+\frac{(1+2\lm)y}{2}}}_{L^2_{p,q,y}} \\
		&\ls \frac{1}{|\lm|} \|h_1\|_{L^2(\R^2)} \|h_2\|_{L^2(\R^2)} \|f\|_{L^2(\R)} \|g\|_{L^2(\R)}.
	\end{align}
	Therefore, we obtain
	\begin{align}
		\SA &\ls \int_{|\lm|\ge 1/2} \frac{d\lm }{|\lm|^2} \|h_1\|_{L^2(\R^2)} \|h_2\|_{L^2(\R^2)} \|f\|_{L^2(\R)} \|g\|_{L^2(\R)} \\
		 &\ls \|h_1\|_{L^2(\R^2)} \|h_2\|_{L^2(\R^2)} \|f\|_{L^2(\R)} \|g\|_{L^2(\R)}.
	\end{align}

	\noindent\textbf{Step 2: Estimate of $\SB$.}
	Set $\mu:=1+\lm$. Changing variables by	$(u,v,w)=(y+p-q,p,\mu q)$, we have
	\begin{equation}
		I(\lm)= \frac{1}{\mu} \iiint_{\R^3} dwdudv
		         f(u-v) h_1\K{v,u-\frac{v}{2}+w} 
		         \ov{g\K{u+\frac{w}{\mu}}} h_2\K{\frac{w}{\mu},u- \mu v + \frac{w}{2\mu}}.
	\end{equation}
	Define, only for the present calculation,
	\begin{align*}
		F(u,v,w)
		&:= f(u-v)h_1\K{v,u-\frac{v}{2}+w},\\
		G(u,v,w)
		&:= g(u+w) \ov{h_2\K{w,u-v + \frac{w}{2}}}.
	\end{align*}
Define $\mu=e^s$ and $(\CU_s G)(u,v,w):=G(u,\mu v, \mu^{-1} w) = G(u,e^sv, e^{-s}w)$.
Then, we have
	\begin{equation}
		\SB = \K{\int_{\log(1/2)}^{\log(1-\ep)} + \int_{\log(1+\ep)}^{\log(3/2)}} \frac{\lg \CU_s G | F \rg_{L^2(\R^3)}}{e^s - 1} ds
	\end{equation}
	and hence
\begin{align}
	|\SB|
	&\le \No{\K{\int_{\log(1/2)}^{\log(1-\ep)} + \int_{\log(1+\ep)}^{\log(3/2)}} \frac{\CU_s G}{e^s-1} ds}_{L^2(\R^3)}
	\|F\|_{L^2(\R^3)}.
\end{align}
To complete the proof, it suffices to prove
\begin{align}
	\sup_{0<\ep\ll 1}\No{\K{\int_{\log(1/2)}^{\log(1-\ep)} + \int_{\log(1+\ep)}^{\log(3/2)}} \frac{\CU_s G}{e^s-1} ds}_{L^2(\R^3)}
	\ls \|G\|_{L^2(\R^3)}
\end{align}
because
\begin{equation}
	\|F\|_{L^2(\R^3)} = \|f\|_{L^2(\R)} \|h_1\|_{L^2(\R^2)},
	\quad \|G\|_{L^2(\R^3)} = \|g\|_{L^2(\R)} \|h_2\|_{L^2(\R^2)}.
\end{equation}
For signs $\si_1,\si_2\in\{\pm\}$, define
	\[
	H_{\si_1,\si_2}(u,v,w)
	:=
	e^{(v+w)/2}
	H(u,\si_1 e^{v},\si_2 e^{w}).
	\]
Then,
	\begin{align}
		&\|H\|_{L^2(\R^3)}^2
		= \sum_{\si_1,\si_2\in\{\pm\}} \|H(u,\si_1 v,\si_2 w)\|^2_{L^2_{u,v,w}(\R\times\R_+\times\R_+)}
		=\sum_{\si_1,\si_2\in\{\pm\}}
		\|H_{\si_1,\si_2}\|_{L^2(\R^3)}^2. \label{eq:decomposition}\\
		&(\CU_s G)_{\si_1,\si_2}(u,v,w)
		=
		G_{\si_1,\si_2}(u,v+s,w-s). \label{eq:change of variable}
	\end{align}
By \eqref{eq:decomposition} and \eqref{eq:change of variable}, it follows that
\begin{align}
	&\No{\K{\int_{\log(1/2)}^{\log(1-\ep)} + \int_{\log(1+\ep)}^{\log(3/2)}} \frac{(\CU_s G)(u,v,w)}{e^s-1} ds}_{L^2_{u,v,w}}^2 \\
	&\quad = \sum_{\si_1,\si_2\in\{\pm\}} \No{ \K{\int_{\log(1/2)}^{\log(1-\ep)} + \int_{\log(1+\ep)}^{\log(3/2)}} \frac{ G_{\si_1,\si_2}(u,v+s,w-s)}{e^s-1} ds}_{L^2_{u,v,w}}^2 \\
	&\quad = \sum_{\si_1,\si_2\in\{\pm\}} \No{ \K{\int_{\log(1/2)}^{\log(1-\ep)} + \int_{\log(1+\ep)}^{\log(3/2)}} \frac{ e^{is(\xi-\et)}}{e^s-1} ds \CG_{\si_1,\si_2}(u,\xi,\et)}_{L^2_{u,\xi,\et}}^2,
\end{align}
where $\CG_{\si_1,\si_2} := \CF_{v\to \xi,w\to\et} G_{\si_1,\si_2}$.
If 
\begin{equation}\label{eq:C_* bounded}
	C_* := \sup_{\ze \in\R,\;0<\ep\ll 1}  \abs{\K{\int_{\log(1/2)}^{\log(1-\ep)} + \int_{\log(1+\ep)}^{\log(3/2)}} \frac{ e^{i\ze s}}{e^s-1} ds}^2 < \I,
\end{equation}
then we obtain
\begin{align}
	&\sum_{\si_1,\si_2\in\{\pm\}} \No{\K{\int_{\log(1/2)}^{\log(1-\ep)} + \int_{\log(1+\ep)}^{\log(3/2)}} \frac{ e^{is(\xi-\et)}}{e^s-1} ds \CG_{\si_1,\si_2}(u,\xi,\et)}_{L^2_{u,\xi,\et}}^2 \\
	&\qquad \le C_* \sum_{\si_1,\si_2\in\{\pm\}} \|\CG_{\si_1,\si_2}\|_{L^2_{u,\xi,\et}}^2
	        = C_* \sum_{\si_1,\si_2\in\{\pm\}} \|G_{\si_1,\si_2}\|_{L^2_{u,v,w}}^2
	        = C_* \|G\|_{L^2_{u,v,w}}^2,
\end{align}
which is the desired estimate.
Therefore, it remains only to prove \eqref{eq:C_* bounded}.
We have
\begin{align}
	&\abs{\K{\int_{\log(1/2)}^{\log(1-\ep)} + \int_{\log(1+\ep)}^{\log(3/2)}} \frac{ e^{i\ze s}}{e^s-1} ds} \\
	&\quad \le \abs{\int_{\log(1/2)}^{\log(2/3)}\frac{ e^{i\ze s}}{e^s-1} ds }
	           + \abs{\int_{-\log(1+\ep)}^{\log(1-\ep)}\frac{ e^{i\ze s}}{e^s-1} ds }
	           + \abs{\K{\int_{-\log(3/2)}^{-\log(1+\ep)} + \int_{\log(1+\ep)}^{\log(3/2)}} \frac{ e^{i\ze s}}{e^s-1} ds} \\
	&\quad = \SC_1 + \SC_2 + \SC_3.
\end{align}
It is easy to see that
\begin{equation}
	\SC_1 \ls 1.
\end{equation}
For $\SC_2$, we have
\begin{equation}
	\SC_2 \le \abs{\int_{\log(1-\ep)}^{-\log(1+\ep)}\frac{1}{s} ds } + \int_{\log(1-\ep)}^{-\log(1+\ep)} \abs{\frac{1}{e^s-1}-\frac{1}{s}} ds \ls \abs{\log\K{\frac{-\log(1+\ep)}{\log(1-\ep)}}}+1 \ls 1
\end{equation}
if $0<\ep \ll 1$.
Finally, for $\SC_3$, we can compute
\begin{align}
	\SC_3
	&\ls 1+\abs{\K{\int_{-\log(3/2)}^{-\log(1+\ep)} + \int_{\log(1+\ep)}^{\log(3/2)}} \frac{ e^{i\ze s}}{s} ds}\sim 1 + \abs{\int_{\log(1+\ep)}^{\log(3/2)} \frac{\sin(\ze s)}{s} ds}  \\
	&\le 1+ \sup_{0<R_1<R_2<\I} \abs{\int_{R_1}^{R_2} \frac{\sin(s)}{s} ds} \ls 1.
\end{align}
Collecting the above estimates, we obtain \eqref{eq:C_* bounded}.
\end{proof}

\subsection{Higher order terms}\label{subsec:higher order}
In this section, we prove Theorem \ref{th:multilinear 1} for general $n \ge 3$ by induction.
As a preparation, we give an explicit partition formula (Lemma \ref{lem:equal subdivision}).
First, let $s<c<t$ and $n \ge 2$. Then,
\begin{align}\label{eq:subdivision}
	W_{V_1,\dots,V_n}^{(n)}(t,s)
	=
	W_{V_1,\dots,V_n}^{(n)}(t,c)
	+W_{V_1,\dots,V_n}^{(n)}(c,s)
	+\sum_{k=1}^{n-1}
	W_{V_1,\dots,V_k}^{(k)}(t,c)
	W_{V_{k+1},\dots,V_n}^{(n-k)}(c,s)      
\end{align}
for all $V_1,\dots,V_n\in \CL(\R)$ because, for almost all $(t_1,\dots,t_n)$, we have
\begin{align}\label{eq:simplex split}
	\II_{\{t>t_1>\cdots>t_n>s\}}
	&=
	\II_{\{t>t_1>\cdots>t_n>c>s\}}
	+ \II_{\{t>c>t_1>\cdots>t_n>s\}}+\sum_{k=1}^{n-1}
	\II_{\{t>t_1>\cdots>t_k>c\}}
	\II_{\{c>t_{k+1}>\cdots>t_n>s\}}.
\end{align}

\begin{lemma}[Decomposition identity]\label{lem:equal subdivision}
	Let $n\ge 2$ and $m\ge 1$. Let $s<t$ and $I=[s,t]$.
	Let
	\[
	s=t_0^m<t_1^m<\cdots<t_{2^m-1}^m<t_{2^m}^m=t
	\]
	be a partition of $I$.  For $0\leq\ell\leq m$ and
	$0\leq j\leq 2^\ell$, define coarser partitions by
	\[
	t_j^\ell:=t_{2^{m-\ell}j}^m.
	\]
	Then, for any $V_1,\dots,V_n \in \CL(\R)$, 
	\begin{equation}\label{eq:binary decomposition}
		\begin{aligned}
			W_{V_1,\ldots,V_n}^{(n)}(t,s)
			&=\sum_{j=0}^{2^m-1}
			W_{V_1,\ldots,V_n}^{(n)}(t_{j+1}^m,t_j^m) \\
			&\quad+\sum_{\ell=0}^{m-1}\sum_{j=0}^{2^\ell-1}
			\sum_{k=1}^{n-1}
			W_{V_1,\ldots,V_k}^{(k)}
			(t_{2j+2}^{\ell+1},t_{2j+1}^{\ell+1})
			W_{V_{k+1},\ldots,V_n}^{(n-k)}
			(t_{2j+1}^{\ell+1},t_{2j}^{\ell+1}).
		\end{aligned}
	\end{equation}
\end{lemma}

\begin{proof}
We prove \eqref{eq:binary decomposition} by induction on $m$.
First, the observation \eqref{eq:subdivision} gives \eqref{eq:binary decomposition} for $m=1$.
Assume that \eqref{eq:binary decomposition} holds for $1, \dots, m$.
We now consider the case $m+1$.
Let $t_j^m=t_{2j}^{m+1}$ for $j \in \{0,1,\dots,2^m\}$.
Applying the induction hypothesis, we obtain \eqref{eq:binary decomposition}.
For each $j\in\{0,\ldots,2^m-1\}$, we have
	\[
	t_j^m=t_{2j}^{m+1}
	<t_{2j+1}^{m+1}
	<t_{2j+2}^{m+1}=t_{j+1}^m.
	\]
	We may therefore apply \eqref{eq:subdivision} to the corresponding term in the
	first sum in \eqref{eq:binary decomposition}, with
	\[
	a=t_j^m=t^{m+1}_{2j},
	\qquad
	c=t_{2j+1}^{m+1},
	\qquad
	b=t_{j+1}^m=t^{m+1}_{2j+2}.
	\]
	It follows that
	\begin{equation}
	\label{eq:split-one-coarse-interval}
	\begin{aligned}
		W_{V_1,\dots,V_n}^{(n)}(t_{j+1}^m,t_j^m)
		&=W_{V_1,\dots,V_n}^{(n)}
		(t_{2j+2}^{m+1},t_{2j+1}^{m+1})
		+W_{V_1,\dots,V_n}^{(n)}
		(t_{2j+1}^{m+1},t_{2j}^{m+1})\\
		&\quad+\sum_{k=1}^{n-1}
		W_{V_1,\dots,V_{k}}^{(k)}
		(t_{2j+2}^{m+1},t_{2j+1}^{m+1})
		W_{V_{k+1},\dots,V_n}^{(n-k)}
		(t_{2j+1}^{m+1},t_{2j}^{m+1}).
	\end{aligned}
\end{equation}
	Summing the first two terms on the right-hand side of
	\eqref{eq:split-one-coarse-interval} over $j$ gives
	\begin{align*}
		\sum_{j=0}^{2^m-1}
		\Big[
		W_{V_1,\ldots,V_n}^{(n)}
		(t_{2j+2}^{m+1},t_{2j+1}^{m+1})
		+W_{V_1,\ldots,V_n}^{(n)}
		(t_{2j+1}^{m+1},t_{2j}^{m+1})
		\Big] =\sum_{j=0}^{2^{m+1}-1}
		W_{V_1,\ldots,V_n}^{(n)}(t_{j+1}^{m+1},t_j^{m+1}).
	\end{align*}
	Moreover, the remaining terms in
	\eqref{eq:split-one-coarse-interval}, summed over $j$, are
	\[
	\sum_{j=0}^{2^m-1}\sum_{k=1}^{n-1}
	W_{V_1,\dots,V_{k}}^{(k)}
	(t_{2j+2}^{m+1},t_{2j+1}^{m+1})
	W_{V_{k+1},\dots,V_n}^{(n-k)}
	(t_{2j+1}^{m+1},t_{2j}^{m+1}),
	\]
	which is precisely the contribution with $\ell=m$ in the second sum
	in \eqref{eq:binary decomposition}.
	Collecting the above identities, we obtain \eqref{eq:binary decomposition} for $m+1$, and the induction is complete.
\end{proof}

Now, we give a complete proof of Theorem \ref{th:multilinear 1}.
\begin{proof}[Proof of Theorem \ref{th:multilinear 1}]
	\noindent \textbf{Step 0: Setup of the proof.}
	Note that we can assume $s<t$ because
	$$
	\No{W_{V_1,\dots,V_n}^{(n)}(t,s)}_{\CB} = \No{W_{\ov{V_n},\dots,\ov{V_1}}^{(n)}(s,t)}_{\CB}.
	$$
	Let $C_1$ and $C_2$ be constants in Lemmas \ref{lem:first order term} and \ref{lem:second order term}.
	Choose sufficiently small $0<\ep\ll 1$ and set
	$$A := \max\K{1,\frac{C_1}{\ep},\sqrt{\frac{4C_2}{\ep}}}.$$
	Then, Lemmas \ref{lem:first order term} and \ref{lem:second order term} imply that \eqref{eq:strong induction 1} holds for $n=1,2$.
	Let $n\ge3$ and suppose \eqref{eq:strong induction 1} holds for $1,2,\dots,n-1$.
	Let $m\ge 1$.
	Let $V_1,\dots,V_n \in\CL(I)$.
	Then, there exists $R>1$ such that $\supp_{t,\xi} \wh{V_j}(t,\xi) \subset I\times \{1/R<|\xi|<R\}$ for $j=1,\dots,n$.
	There exists a partition
	\[
	s=t_0^m<t_1^m<\cdots<t_{2^m-1}^m<t_{2^m}^m=t
	\]
	such that
	\begin{equation}\label{eq:subdivision of interval}
		\|V_1\|_{L^2_t(t_j^m,t_{j+1}^m;\dH^{-1/2}_x)} = 2^{-m/2} \|V_1\|_{L^2_t(I;\dH^{-1/2}_x)}
		\quad \text{for all} \quad j=0,1,\dots,2^m-1.
	\end{equation}
	By Lemma \ref{lem:equal subdivision}, we have 
	\begin{multline}
		\No{W_{V_1,\dots,V_n}^{(n)}(t,s)}_{\CB} 
		\le \sum_{j=0}^{2^m-1}
		\No{W_{V_1,\dots,V_n}^{(n)}
			(t^m_{j+1},t^m_j)}_{\CB}\\
		+
		\sum_{\ell=0}^{m-1}\sum_{j=0}^{2^\ell-1}\sum_{k=1}^{n-1}
		\No{W_{V_1,\dots,V_{k}}^{(k)}(t^{\ell+1}_{2j+2},t^{\ell+1}_{2j+1})}_{\CB}
		\No{W_{V_{k+1},\dots,V_n}^{(n-k)} (t^{\ell+1}_{2j+1},t^{\ell+1}_{2j})}_{\CB}
		=: \SA_m + \SB_m.
	\end{multline}
	\noindent \textbf{Step 1: Estimate of $\SA_m$.} For $\SA_m$, by the H\"older inequality and \eqref{eq:subdivision of interval}, we have 
	\begin{align}
		\SA_m &\le \sum_{j=0}^{2^m-1} \prod_{k=1}^n \|V_k\|_{L^1_t(t_j^{m},t_{j+1}^m;L^\I_x)} \\
		&\le \sum_{j=0}^{2^m-1} |t_{j+1}^m - t_j^m|^{n/2} \prod_{k=1}^n
		\K{\int_{|\xi|\le R}|\xi|d\xi}^{1/2}\|V_k\|_{L^2_t(t_j^{m},t_{j+1}^m;\dH_x^{-1/2})} \\
		&= 2^{-m/2} C_R^{n/2}\sum_{j=0}^{2^m-1} |t_{j+1}^m - t_j^m|^{n/2} \prod_{i=2}^n \|V_i\|_{L^2_t(I;\dH^{-1/2}_x)},
	\end{align}
	where $C_R := \int_{|\xi|\le R} |\xi|d\xi$.
	Since
	$$\sum_{j=0}^{2^m-1} |t_{j+1}^m - t_j^m|^{n/2}
	\le |I|^{n/2-1} \sum_{j=0}^{2^m-1} |t_{j+1}^m - t_j^m|\le |I|^{n/2},$$
	we obtain
	\begin{equation}
		\SA_m \ls 2^{-m/2} \to 0
	\end{equation}
	as $m\to\I$.
	
	\noindent \textbf{Step 2: Estimate of $\SB_m$.}
	For $\SB_m$, by the induction hypothesis and \eqref{eq:subdivision of interval},
	\begin{align}
		\SB_m
		&\le \sum_{\ell=0}^{m-1}\sum_{j=0}^{2^\ell-1}\sum_{k=1}^{n-1} 
		\No{W_{V_1,\ldots,V_k}^{(k)}(t^{\ell+1}_{2j+2},t^{\ell+1}_{2j+1})}_{\CB}
		\No{W_{V_{k+1},\ldots,V_n}^{(n-k)} (t^{\ell+1}_{2j+1},t^{\ell+1}_{2j})}_{\CB} \\
		&\le \ep^2 A^n \sum_{\ell=0}^{m-1}\sum_{j=0}^{2^\ell-1}\sum_{k=1}^{n-1} 
		\frac{1}{(n-k)^2 k^2}
		\prod_{i=1}^{k} \|V_i\|_{L^2_t(t^{\ell+1}_{2j+1},t^{\ell+1}_{2j+2};\dH^{-1/2}_x)}
		\prod_{i'=k+1}^{n} \|V_{i'}\|_{L^2_t(t^{\ell+1}_{2j},t^{\ell+1}_{2j+1};\dH^{-1/2}_x)} \\
		&\le \ep^2 A^n \sum_{\ell=0}^{m-1}\sum_{k=1}^{n-1} 
	\frac{2^{-(\ell+1)/2}}{(n-k)^2 k^2} \|V_1\|_{L^2_t(I;\dH^{-1/2}_x)} \\
	    &\qquad \cdot \sum_{j=0}^{2^\ell-1} 
	    \prod_{i=2}^{k} \|V_i\|_{L^2_t(t^{\ell+1}_{2j+1},t^{\ell+1}_{2j+2};\dH^{-1/2}_x)}
	    \prod_{i'=k+1}^{n} \|V_{i'}\|_{L^2_t(t^{\ell+1}_{2j},t^{\ell+1}_{2j+1};\dH^{-1/2}_x)},
	\end{align}
	where we used
	\begin{align}
		\|V_1\|_{L^2_t(t^{\ell+1}_{2j+1},t^{\ell+1}_{2j+2};\dH^{-1/2}_x)}
		= 2^{(m-\ell-1)/2}2^{-m/2}\|V_1\|_{L^2_t(I;\dH^{-1/2}_x)} = 2^{-(\ell+1)/2} \|V_1\|_{L^2_t(I;\dH^{-1/2}_x)}.
	\end{align}
	Applying the H\"older inequality to the summation over $j$, we obtain
	\begin{align}
		&\sum_{j=0}^{2^\ell-1} 		
		\prod_{i=2}^{k} \|V_i\|_{L^2_t(t^{\ell+1}_{2j+1},t^{\ell+1}_{2j+2};\dH^{-1/2}_x)}
		\prod_{i'=k+1}^{n} \|V_{i'}\|_{L^2_t(t^{\ell+1}_{2j},t^{\ell+1}_{2j+1};\dH^{-1/2}_x)} \\
		&\quad \le 		
		\prod_{i=2}^{k}\K{ \sum_{j=0}^{2^\ell-1} \|V_i\|_{L^2_t(t^{\ell+1}_{2j+1},t^{\ell+1}_{2j+2};\dH^{-1/2}_x)}^{n-1}}^{1/(n-1)}
		\prod_{i'=k+1}^{n} \K{\sum_{j=0}^{2^\ell-1} \|V_{i'}\|^{n-1}_{L^2_t(t^{\ell+1}_{2j},t^{\ell+1}_{2j+1};\dH^{-1/2}_x)}}^{1/(n-1)} \\ 
		&\quad \le 
		\prod_{i=2}^{k}\K{ \sum_{j=0}^{2^\ell-1} \|V_i\|^2_{L^2_t(t^{\ell+1}_{2j+1},t^{\ell+1}_{2j+2};\dH^{-1/2}_x)}}^{1/2}
		\prod_{i'=k+1}^{n} \K{\sum_{j=0}^{2^\ell-1} \|V_{i'}\|^2_{L^2_t(t^{\ell+1}_{2j},t^{\ell+1}_{2j+1};\dH^{-1/2}_x)}}^{1/2} \\
		&\quad \le \prod_{i=2}^n \|V_i\|_{L^2_t(I;\dH^{-1/2}_x)}. 
	\end{align}
	Moreover, we have
	\begin{equation}
		\sum_{k=1}^{n-1} \frac{1}{k^2(n-k)^2} \le \frac{C}{n^2},
	\end{equation}
	where $C>0$ is an absolute constant. 
	Hence, we have
	\begin{equation}
		\SB_m \le
		\frac{(1+\sqrt{2}) C\ep^2 A^n}{n^2} \prod_{j=1}^n \|V_j\|_{L^2_t(I;\dH^{-1/2}_x)}
		\le \frac{\ep A^n}{n^2} \prod_{j=1}^n \|V_j\|_{L^2_t(I;\dH^{-1/2}_x)},
	\end{equation}
	where we used $\ep\ll 1$ to get $(1+\sqrt{2}) C\ep \le 1$.
	
	\noindent \textbf{Step 3: Conclusion.}
	Collecting the above estimates, we have
	\begin{equation}
		\No{W_{V_1,\dots,V_n}^{(n)}(t,s)}_{\CB} 
		\le \limsup_{m\to\I} \SA_m + \limsup_{m\to\I}\SB_m
		\le \frac{\ep A^n}{n^2} \prod_{j=1}^n \|V_j\|_{L^2_t(I;\dH^{-1/2}_x)}.
	\end{equation}
\end{proof}

\section{Construction of perturbed propagators and Strichartz estimates}\label{sec:SV Stri}
In this section, we first construct the perturbed propagator $S_V$ and inverse wave operators $W_V^\pm$ for arbitrary $V \in L^2_t(\R;\dH^{-1/2}_x)$ (proof of Theorem \ref{th:construction of SV}).
We then prove two types of Strichartz estimates for $S_V$ when $V$ is small (proof of Theorem \ref{th:Stri}).

\subsection{Construction of propagators and inverse wave operators}\label{sec:SV}
Now, we prove Theorem \ref{th:construction of SV} by applying the multilinear estimates obtained in the previous section.
First, we give a precise definition of $S_V$ when $V$ is small.
\begin{definition}[Definition of $S_V$ with a smallness assumption]\label{def:UV}
	Let $A>0$ be the constant in Theorem \ref{th:multilinear 1}.
	Let $I\subset \R$ be an interval.
	Assume that $V \in L^2_t(I;\dH^{-1/2}_x(\R))$ satisfies
	\begin{equation}\label{eq:smallness}
		\de_0 := A\|V\|_{L^2_t(I;\dH^{-1/2}_x(\R))} < 1.
	\end{equation}
	Then, we define $(S_V(t,s))_{t,s\in I}$ as
	\begin{equation}\label{eq:UV definition}
		S_V(t,s) = S(t)\sum_{n\ge 0} W_V^{(n)}(t,s)S(-s).
	\end{equation}
	Note that $S_V(s,s) = \Id_{L^2_x(\R)}$ for all $s\in I$.
	When $0\in I$, to shorten the notation, we write $S_V(t):=S_V(t,0)$.
\end{definition}
\begin{remark}[Bound of operator norms]
	By Theorem \ref{th:multilinear 1} and the smallness assumption for $V$, the infinite series on the right-hand side of \eqref{eq:UV definition} absolutely converges in $\CB$. Moreover, we have
	\begin{equation}\label{eq:bound of UV}
		\sup_{t,s\in I} \|S_V(t,s)\|_{\CB} \le \frac{2}{1-\de_0}.
	\end{equation}
\end{remark}

\begin{remark}[Coincidence with the classical definition of $S_V$ for smooth $V$]\label{rmk:coincides}
	When a very nice function $V$ satisfies \eqref{eq:smallness}, we can define $S_V$ in two different ways. 
	On the one hand, we can define $S_V$ as the solution to the Duhamel formula \eqref{eq:Duhamel}.
	On the other hand, we can define $S_V$ by Definition \ref{def:UV}.
	However, these two definitions coincide.
	Indeed, if we repeatedly use the Duhamel formula \eqref{eq:Duhamel}, then we obtain \eqref{eq:UV definition}.
\end{remark}

\begin{remark}[Finite-time wave operator]\label{rmk:WV}
	Define the finite-time wave operator as $W_V(t,s):=S(-t)S_V(t,s)S(s)$. Then, we have
	\begin{equation}\label{eq:WV}
		W_V(t,s) = \sum_{n\ge 0} W_V^{(n)}(t,s).
	\end{equation}
	To shorten the notation, we write $W_V(t):=W_V(t,0)$.
\end{remark}

The propagator defined in Definition \ref{def:UV} is stable in the following sense.
\begin{lemma}[Stability estimate]\label{lem:difference estimate}
	Let $I \subset \R$ be an interval.
	Assume that $V$ and $V'$ satisfy
	$$\de_1:= \max\K{A\|V\|_{L^2_t(I;\dH^{-1/2}_x(\R))}, A\|V'\|_{L^2_t(I;\dH^{-1/2}_x(\R))}}< 1.$$
	Then, we have
	\begin{equation}\label{eq:small stability}
		\sup_{t,s\in I}\No{S_V(t,s)-S_{V'}(t,s)}_{\CB}
		\le \frac{\ep A}{1-\de_1} \|V-V'\|_{L^2_t(I;\dH^{-1/2}_x(\R))},
	\end{equation}
	where $\ep$ and $A$ are constants in Theorem \ref{th:multilinear 1}. 
\end{lemma}

\begin{remark}\label{rmk:natural}
By Lemma \ref{lem:difference estimate}, we can see that the construction of $S_V$ in Definition \ref{def:UV} is natural in the following sense.
If $V$ satisfies \eqref{eq:smallness}, then we can choose a sequence of very nice functions $V_j$ such that $V_j \to V$ in $L^2_t(I;\dH^{-1/2}_x(\R))$.
Then, by Lemma \ref{lem:difference estimate},
$$S_{V_j}(t,s)\to S_V(t,s) \quad \text{as} \quad j\to\I.$$
Therefore, every propagator $S_V$ constructed in Definition \ref{def:UV} is the unique limit of classical propagators obtained from smooth approximations of $V$.
\end{remark}

\begin{proof}[Proof of Lemma \ref{lem:difference estimate}]
	By Theorem \ref{th:multilinear 1}, we obtain
	\begin{align}
		\|S_V(t,s)-S_{V'}(t,s)\|_{\CB}
		&= \No{\sum_{n\ge 0} W_{V,\dots,V}^{(n)}(t,s)- \sum_{n\ge 0} W_{V',\dots,V'}^{(n)}(t,s)}_{\CB} \\
		&\le \sum_{n\ge 1} \No{W^{(n)}_{V,\dots,V}(t,s) - W_{V',\dots,V'}^{(n)}(t,s)}_{\CB} \\
		&\le \sum_{n\ge 1} \Big(\|W^{(n)}_{V-V',V,\dots,V}(t,s)\|_{\CB} +\cdots + \|W^{(n)}_{V',\dots,V',V-V'}(t,s)\|_{\CB} \Big) \\
		&\le \ep A \|V-V'\|_{L^2_t(I;\dH^{-1/2}_x)} \sum_{n\ge 1} \frac{\de_1^{n-1}}{n}
		\le \frac{\ep A}{1-\de_1} \|V-V'\|_{L^2_t(I;\dH^{-1/2}_x)}.
	\end{align}
	Since the above estimates are uniform in $(t,s)$, we obtain \eqref{eq:small stability}.
\end{proof}

Next, we prove some basic properties of $S_V(t,s)$.
\begin{lemma}[Basic properties of $S_V$ with small $V$]\label{lem:basic UV}
Let $I \subset \R$ be an interval.
Assume that $V \in L^2_t(I;\dH^{-1/2}_x(\R))$ satisfies
	\begin{equation}\label{eq:V condition}
		A \|V\|_{L^2_t(I;\dH^{-1/2}_x(\R))}\le \tw,
	\end{equation}
where $A$ is the constant in Theorem \ref{th:multilinear 1}.
	Then, the family of bounded operators $(S_V(t,s))_{t,s\in I}$ satisfies the following:
	\begin{enumerate}[$(i)$]
		\item For all $t\in I$, $S_V(t,t) = \Id_{L^2_x(\R)}$.
		\item For all $t,s,r\in I$, $S_V(t,s)S_V(s,r) = S_V(t,r)$. In particular, $S_V(t,s)^{-1} = S_V(s,t)$.
		\item Both $t\mapsto S_V(t,s)$ and $s\mapsto S_V(t,s)$ are strongly continuous in $L^2_x(\R)$.
	\end{enumerate}
\end{lemma}

\begin{proof}
	First, we have $(i)$ by the definition.
	When $V \in \CL(I)$, the propagator $S_V(t,s)$ for \eqref{eq:S} exists in a classical sense.
	More precisely, $S_V(t,s)$ satisfies
	\begin{equation}
		S_V(t,s)=S(t-s) - i\int_s^t S(t-\ta)V(\ta)S_V(\ta,s)d\ta.
	\end{equation}
	Hence, the standard argument implies $(ii)$ and $(iii)$.
	
	Let $V \in L^2_t(I;\dH^{-1/2}_x)$ be such that \eqref{eq:V condition}.
	Then, we can choose $V_j\in \CL(I)$ such that $V_j \to V$ in $L^2_t(I;\dH^{-1/2}_x)$.
	By Lemma \ref{lem:difference estimate}, we know 
	\begin{equation}
		\lim_{j\to \I} \sup_{t,s\in I} \No{S_{V_j}(t,s) -S_{V}(t,s)}_{\CB} = 0.
	\end{equation}
	Therefore, $(S_V(t,s))_{t,s\in I}$ also satisfies $(ii)$ and $(iii)$.
\end{proof}

Finally, we give a definition of $S_V(t,s)$ without any size restriction on $V$.
\begin{definition}[Definition of $S_V$ without a smallness assumption]\label{def:UV unconditional}
	Let $I\subset \R$ be an interval.
	Assume $V \in L^2_t(I;\dH^{-1/2}_x(\R))$.
	We define $S_V(t,t)=\Id_{L^2_x(\R)}$ for all $t\in I$.
	When $s<t$, we can choose a partition of $[s,t]$, $s= r_0 < r_1 < \cdots < r_N = t$, such that
	$$\sup_{n=0,\dots,N-1} A \|V\|_{L^2_t(r_n,r_{n+1};\dH^{-1/2}_x(\R))}\le \tw,$$
	where $A>0$ is the constant in Theorem \ref{th:multilinear 1}.
	Then, we define $S_V(t,s)$ by
	\begin{equation}\label{eq:UV definition unconditional}
		S_V(t,s) = S_V(t,r_{N-1})\cdots S_V(r_1,s).
	\end{equation}	
When $t<s$, we define
	\begin{equation}
		S_V(t,s) = S_{\ov{V}}(s,t)^*.
	\end{equation}
\end{definition}

\begin{remark}[Well-definedness of $S_V$ for large $V$]\label{rmk:well-defined}
	Definition \ref{def:UV unconditional} is well-defined.
	Let $s=r_0<r_1<\cdots<r_{N} = t$ and $s=r'_0< \cdots < r'_{N'} = t$ be two different partitions of $[s,t]$.
	Then, we take a common refinement
	$$s=r''_0 <r''_1 < \cdots < r''_{N''} = t.$$
	Therefore, by Lemma \ref{lem:basic UV}, we have 
	\begin{equation}
		S_V(t,r_{N-1})\cdots S_V(r_1,s)
		= S_V(t,r''_{N''-1})\cdots S_V(r''_1,s) = S_V(t,r'_{N'-1})\cdots S_V(r'_1,s).
	\end{equation}
\end{remark}

Finally, we prove Theorem \ref{th:construction of SV}.
\begin{proof}[Proof of Theorem \ref{th:construction of SV}]
	The properties $(i)$, $(ii)$, and $(iii)$ follow from Definition \ref{def:UV unconditional} and Lemma \ref{lem:basic UV}.
	Next we prove $(iv)$.
	It suffices to consider the case $s<t$ because we have the identity $S_V(t,s)=S_{\ov{V}}(s,t)^*$.
	Let $s=r_0 < \cdots < r_N =t$ be a partition as in Definition \ref{def:UV unconditional}.
	Then, by \eqref{eq:bound of UV}, we obtain
	\begin{equation}
		\|S_V(t,s)\|_{\CB} \le \prod_{j=0}^{N-1} \|S_V(r_{j+1},r_{j})\|_{\CB} \le 4^N.
	\end{equation}
	Since we can take $N$ such that
	$$N \ls 1+\|V\|_{L^2_t(I;\dH^{-1/2}_x)}^2,$$
	we obtain \eqref{eq:SV bound}.

Finally, we prove $(v)$. We only consider $t\to\I$ because we can deal with $t\to-\I$ in the same way.
Since $V$ belongs to $L^2_t(\R;\dH^{-1/2}_x)$, there exists a sufficiently large $T>0$ such that
\begin{equation}
	A \|V\|_{L^2_t(T,\I;\dH^{-1/2}_x)} \le \tw,
\end{equation}
where $A$ is the absolute constant in Theorem \ref{th:multilinear 1}.
Then, by Theorem \ref{th:multilinear 1} (see also Remark \ref{rmk:WV}) and Theorem \ref{th:construction of SV} $(iv)$, we have
\begin{align}
	\No{S(-t)S_V(t,0) - S(-s)S_V(s,0)}_{\CB}
	&\le \|S_V(T,0)\|_{\CB}  \No{S(-t)S_V(t,T) - S(-s)S_V(s,T)}_{\CB}  \\
	&\le \Ph\K{\|V\|_{L^2_t(\R;\dH^{-1/2}_x)}} \sum_{n\ge 1} \No{W_V^{(n)}(t,T) - W_V^{(n)}(s,T)}_{\CB}
\end{align}
for $T<s<t<\I$.
Note that, by the definition of $W_V^{(n)}$, we have
$$W_V^{(n)}(t,T) - W_V^{(n)}(s,T) = W^{(n)}_{V_s, V,\dots,V}(t,T),$$
where $V_s(\ta,x) := V(\ta,x)\II_{(s\le \ta)} $.
Hence, Theorem \ref{th:multilinear 1} implies
\begin{align}
	\sum_{n\ge 1} \No{W_V^{(n)}(t,T) - W_V^{(n)}(s,T)}_{\CB}
	&\le \sum_{n\ge 1} \frac{\ep A^n}{n^2} \|V\|_{L^2_t(s,t;\dH^{-1/2}_x)} \|V\|_{L^2_t(T,\I;\dH^{-1/2}_x)}^{n-1}\\
	&\ls \|V\|_{L^2_t(s,t;\dH^{-1/2}_x)} \to 0\quad \text{as} \quad s,t \to \I.
\end{align}
Therefore, by the completeness of the operator space, we find $W_V^+ \in \CB$ such that
\begin{equation}
	\|S(-t)S_V(t) - W_V^+\|_\CB \to 0\quad \text{as} \quad t \to \I.
\end{equation}
\end{proof}

\subsection{Standard Strichartz estimates for $S_V$}\label{subsec:standard Stri}
The well-known Strichartz estimates for the free Schr\"odinger equation on the line are
\begin{align}
	&\|S(t)\phi\|_{\CX(\R)}  \ls \|\phi\|_{L^2_x(\R)}, \label{eq:free Stri} \\
	&\No{\int_s^t S(t-\ta)f(\ta)d\ta}_{\CX(I)}  \ls \|f\|_{\CX'(I)},
\end{align}
where $I$ is an interval containing $s$ and
\begin{equation}\label{eq:definition of X}
	\begin{aligned}
		&\CX(I):=L^4_t(I;L^\I_x(\R)) \cap L^\I_t(I;L^2_x(\R)),\\
		&\CX'(I):=L^{4/3}_t(I;L^1_x(\R)) + L^1_t(I;L^2_x(\R)).
	\end{aligned}
\end{equation}

We give an analogue for $S_V$.
\begin{lemma}[Generalized Strichartz estimates]\label{lem:standard Stri}
	Let $I$ be an interval and $s\in I$.
	Assume that $V$ satisfies
	$$A\|V\|_{L^2_t(I;\dH^{-1/2}_x(\R))}\le \tw,$$
	where $A>0$ is the constant in Theorem \ref{th:multilinear 1}.
	Then,
	\begin{align}
		&\|S_V(t,s)\phi\|_{\CX(I)} \le C_{\textup{Stri}} \|\phi\|_{L^2_x(\R)}, \label{eq:Stri 1}\\
		&\No{\int_s^t S_V(t,\ta)f(\ta)d\ta}_{\CX(I)} \le C_{\textup{Stri}} \|f\|_{\CX'(I)}, \label{eq:Stri 2}
	\end{align}
	where $C_{\textup{Stri}}>0$ is an absolute constant.
\end{lemma}

\begin{proof}
	It suffices to prove \eqref{eq:Stri 1} because the standard duality argument and the Christ--Kiselev lemma imply \eqref{eq:Stri 2}.
	Let $J:= [s,\I) \cap I$. We only prove
	\begin{align}
		\|S_V(t,s)\phi\|_{\CX(J)} \le C_{\textup{Stri}} \|\phi\|_{L^2_x(\R)}
	\end{align}
	because we can see
	\begin{align}
		\|S_V(t,s)\phi\|_{\CX(I \setminus J)} \le C_{\textup{Stri}} \|\phi\|_{L^2_x(\R)}
	\end{align}
	in a similar manner.
	First, the bound \eqref{eq:bound of UV} implies
	\begin{align}
		\|S_V(t,s)\phi\|_{C_t(I;L^2_x)}
		\le 4 \|\phi\|_{L^2_x}.
	\end{align}
	Next, we consider the $L^4_t L^\I_x$ norm. 
	By Definition \ref{def:UV} and \eqref{eq:free Stri}, 
	\begin{equation}
		\|S_V(t,s)\phi\|_{L^4_t(J;L^\I_x)}
		\le C\|\phi\|_{L^2_x} + \sum_{n \ge 1} \No{S(t)W_V^{(n)}(t,s)S(-s)\phi}_{L^4_t(J;L^\I_x)},
	\end{equation}
	where $C\ge 1$ is an absolute constant.
	Fix an arbitrary $T\in J$.
	Since
	\begin{align}
		\No{S(t)W_{V_1,\dots,V_n}^{(n)}(T,s)S(-s)\phi}_{L^4_t(s,T;L^\I_x)}
		&\le C \No{W_{V_1,\dots,V_n}^{(n)}(T,s)S(-s)\phi}_{L^2_x} \\
		&\le \frac{C \ep }{n^2} A^n \prod_{j=1}^n \|V_j\|_{L^2_t([s,T];\dH^{-1/2}_x)} \|\phi\|_{L^2_x}
	\end{align}
	with an absolute constant $C\ge 1$, 
	the Christ--Kiselev lemma (see \cite{Christ Kiselev 2001}) implies
	\begin{align}\label{eq:multi linear}
		\No{S(t)W_{V_1,\dots,V_n}^{(n)}(t,s)S(-s)\phi}_{L^4_t([s,T];L^\I_x)}
		\le \frac{C' A^n}{n^2} \prod_{j=1}^n \|V_j\|_{L^2_t([s,T];\dH^{-1/2}_x(\R))} \|\phi\|_{L^2_x},
	\end{align}
	where $C'\ge 1$ is an absolute constant.
	Since $T\in J$ was arbitrary, we obtain
	\begin{align}\label{eq:multi linear infty}
		\No{S(t)W_{V_1,\dots,V_n}^{(n)}(t,s)S(-s)\phi}_{L^4_t(J;L^\I_x)}
		\le \frac{C' A^n}{n^2} \prod_{j=1}^n \|V_j\|_{L^2_t(J;\dH^{-1/2}_x(\R))} \|\phi\|_{L^2_x}.
	\end{align}
	Choosing $V_1 =\cdots =  V_n = V$, we obtain
	\begin{equation}\label{eq:V}
		\|S_V(t,s)\phi\|_{L^4_t(J;L^\I_x)}
		\le C\|\phi\|_{L^2_x} + C' \sum_{n \ge 1} A^n \|V\|_{L^2_t(J;\dH^{-1/2}_x)}^n \|\phi\|_{L^2_x}
		\le (C+2C') \|\phi\|_{L^2_x}.
\end{equation}
Hence, we obtain \eqref{eq:Stri 1} by setting $C_{\text{Stri}} := C+2C'$.
\end{proof}

By a similar argument, we can prove the following difference estimate.
\begin{lemma}[Difference estimate]\label{lem:difference SV}
	Let $I$ be an interval and $s\in I$.
	Assume that $V$ and $V'$ satisfy
	$$r:=\max\K{A\|V\|_{L^2_t(I;\dH^{-1/2}_x(\R))}, A\|V'\|_{L^2_t(I;\dH^{-1/2}_x(\R))}}\le \frac{1}{2},$$
	where $A>0$ is the constant in Theorem \ref{th:multilinear 1}.
	Then, we have
	\begin{align}
		&\|S_V(t,s)\phi - S_{V'}(t,s) \phi\|_{\CX(I)} \le C_{\textup{Stri}}' \|V-V'\|_{L^2_t(I;\dH^{-1/2}_x(\R))} \|\phi\|_{L^2_x(\R)}, \label{eq:Stri difference 1}\\
		&\No{\int_s^t S_V(t,\ta)f(\ta)d\ta - \int_s^t S_{V'}(t,\ta)f(\ta)d\ta}_{\CX(I)} \\
		&\qquad \qquad \qquad \qquad \qquad \qquad \qquad\le C_{\textup{Stri}}' \|V-V'\|_{L^2_t(I;\dH^{-1/2}_x(\R))} \|f\|_{\CX'(I)}, \label{eq:Stri difference 2}
	\end{align}
	where $C_{\textup{Stri}}'>0$ is an absolute constant.
\end{lemma}

\begin{proof}
Let $J:= [s,\I) \cap I$. We only estimate the $\|\cdot\|_{\CX(J)}$ norm because we can deal with the $\|\cdot\|_{\CX(I\setminus J)}$ norm in a similar manner.
By Lemma \ref{lem:difference estimate}, we have
\begin{align}
	\|S_V(t,s)\phi - S_{V'}(t,s) \phi\|_{C_t(J;L^2_x)}
	&\le \sup_{t,s\in I} \|S_V(t,s) - S_{V'}(t,s)\|_\CB \|\phi\|_{L^2_x} \\
	&\le 2\ep A \|V-V'\|_{L^2_t(I;\dH^{-1/2}_x)} \|\phi\|_{L^2_x}.
\end{align}
Next, we consider the $L^4_t L^\I_x$ norm. 
By Definition \ref{def:UV},
\begin{align}
	\|S_V(t,s)\phi - S_{V'}(t,s)\phi\|_{L^4_t(J;L^\I_x)}
	&\le \sum_{n \ge 1} \No{S(t)\Dk{W_{V,\dots,V}^{(n)}(t,s) - W_{V',\dots,V'}^{(n)}(t,s)}S(-s)\phi}_{L^4_t(J;L^\I_x)} \\
	&\le \sum_{n \ge 1} \No{S(t) W_{V-V',V,\dots,V}^{(n)}(t,s)S(-s)\phi}_{L^4_t(J;L^\I_x)} + \cdots\\
	&\qquad \cdots + \sum_{n \ge 1} \No{S(t) W_{V',\dots,V',V-V'}^{(n)}(t,s)S(-s)\phi}_{L^4_t(J;L^\I_x)}.
\end{align}
By the estimate \eqref{eq:multi linear} obtained in the proof of Lemma \ref{lem:standard Stri}, we obtain
\begin{align}
	\|S_V(t,s)\phi - S_{V'}(t,s)\phi\|_{L^4_t(J;L^\I_x)}
	&\le \sum_{n \ge 1} \frac{CA r^{n-1}}{n} \|V-V'\|_{L^2_t(J;\dH^{-1/2}_x)}  \|\phi\|_{L^2_x}\\
	&\le 2CA \|V-V'\|_{L^2_t(J;\dH^{-1/2}_x)} \|\phi\|_{L^2_x}.
\end{align}
Thus, we obtain \eqref{eq:Stri difference 1} by setting $C_{\text{Stri}}'=2CA$.

Now, we prove \eqref{eq:Stri difference 2}.
By \eqref{eq:Stri difference 1}, which has already been proved, we have
\begin{align}
	\No{\int_s^t S_V(t,\ta)f(\ta)d\ta - \int_s^t S_{V'}(t,\ta)f(\ta)d\ta}_{\CX(J)}
	&\le \int_J \No{S_V(t,\ta)f(\ta) - S_{V'}(t,\ta)f(\ta)}_{\CX(J)} d\ta \\
	&\le C_{\text{Stri}}'\int_J \|V-V'\|_{L^2_t(J;\dH^{-1/2}_x)} \|f(\ta)\|_{L^2_x} d\ta \\
	&= C_{\text{Stri}}' \|V-V'\|_{L^2_t(J;\dH^{-1/2}_x)} \|f\|_{L^1_t(J;L^2_x)}.
\end{align}
Second, by Theorem \ref{th:construction of SV},
\begin{multline}
	\No{\int_s^t S_V(t,\ta)f(\ta)d\ta - \int_s^t S_{V'}(t,\ta)f(\ta)d\ta}_{\CX(J)}
	\le \No{\K{S_V(t,s) - S_{V'}(t,s)}\int_s^t S_V(s,\ta)f(\ta)d\ta }_{\CX(J)} \\
	 + \No{S_{V'}(t,s)\int_s^t \K{S_V(s,\ta) - S_{V'}(s,\ta)}f(\ta)d\ta}_{\CX(J)} =:\SA + \SB.
\end{multline}
Fix an arbitrary $T\in J$.
Then, by \eqref{eq:Stri difference 1}, we have
\begin{multline}
	\No{\K{S_V(t,s) - S_{V'}(t,s)}\int_s^T S_V(s,\ta)f(\ta)d\ta }_{\CX(J)} \\
	\le C_{\textup{Stri}}' \|V-V'\|_{L^2_t(J;\dH^{-1/2}_x)} \No{\int_s^T S_V(s,\ta)f(\ta)d\ta }_{L^2_x} \\
	\le C_{\textup{Stri}} C_{\textup{Stri}}' \|V-V'\|_{L^2_t(J;\dH^{-1/2}_x)} \|f\|_{L^{4/3}_t(J;L^1_x)}.
\end{multline}
Hence, by the Christ--Kiselev lemma, we have
\begin{equation}
	\SA \le C \|V-V'\|_{L^2_t(J;\dH^{-1/2}_x)} \|f\|_{L^{4/3}_t(J;L^1_x)},
\end{equation}
where $C>0$ is an absolute constant.
Therefore, it remains to estimate $\SB$.
By the Christ--Kiselev lemma, it suffices to prove
\begin{equation}\label{eq:aim}
	\begin{aligned}
		\wt{\SB}&:=\No{S_{V'}(t,s)\int_s^T \K{S_V(s,\ta) - S_{V'}(s,\ta)}f(\ta)d\ta}_{\CX(J)} \\
		&\le C \|V-V'\|_{L^2_t(J;\dH^{-1/2}_x)} \|f\|_{L^{4/3}_t(J;L^1_x)}
	\end{aligned}
\end{equation}
for arbitrary $T \in J$, where $C>0$ is an absolute constant.
By \eqref{eq:Stri 1}, we have
\begin{align}\label{eq:reduce 1}
	\wt{\SB} \le C_{\textup{Stri}} \No{\int_s^T \K{S_V(s,\ta) - S_{V'}(s,\ta)}f(\ta)d\ta}_{L^2_x}.
\end{align}
Let $\phi\in \CS(\R)$. Then, by \eqref{eq:Stri difference 1} and $S_V(t,s)^*=S_{\ov{V}}(s,t)$, we have
\begin{multline}
	\int_\R \ov{\phi(x)} \Dk{\int_s^T \K{S_V(s,\ta) - S_{V'}(s,\ta)}f(\ta)d\ta}(x) dx 
	= \int_s^T \int_\R \ov{\K{S_{\ov{V}}(\ta,s) - S_{\ov{V'}}(\ta,s)} \phi(x) }f(\ta,x) dxd\ta \\
	\le \No{\K{S_{\ov{V}}(\ta,s) - S_{\ov{V'}}(\ta,s)} \phi}_{L^{4}_t(J;L^\I_x)} \|f\|_{L^{4/3}_t(J;L^1_x)} \\
	\le C \|V-V'\|_{L^2_t(J;\dH^{-1/2}_x)} \|\phi\|_{L^2_x} \|f\|_{L^{4/3}_t(J;L^1_x)},
\end{multline}
where $C>0$ is an absolute constant.
Hence, by duality,
\begin{align}\label{eq:reduce 2}
	\No{\int_s^T \K{S_V(s,\ta) - S_{V'}(s,\ta)}f(\ta)d\ta}_{L^2_x}\le C \|V-V'\|_{L^2_t(J;\dH^{-1/2}_x)} \|f\|_{L^{4/3}_t(J;L^1_x)}.
\end{align}
Combining \eqref{eq:reduce 1} and \eqref{eq:reduce 2}, we obtain \eqref{eq:aim}.
\end{proof}

\subsection{Ozawa--Tsutsumi type Strichartz estimates for $S_V$}\label{subsec:Strichartz}
Recall the well-known bilinear Strichartz estimate proved by Ozawa and Tsutsumi in \cite{Ozawa Tsutsumi 1998}.
\begin{theorem}[Ozawa--Tsutsumi]\label{th:Ozawa Tsutsumi}
	For any $\phi,\phi'\in L^2_x(\R)$, we have
	\begin{equation}\label{eq:Ozawa Tsutsumi}
		\No{S(t)\phi \cdot \ov{S(t)\phi'}}_{L^2_t(\R;\dot{H}^{1/2}_x)} \ls \|\phi\|_{L^2_x(\R)} \|\phi'\|_{L^2_x(\R)}.
	\end{equation}
\end{theorem}
In the sequel, we generalize \eqref{eq:Ozawa Tsutsumi} to the perturbed flow $S_V$.
Our first observation is as follows.
\begin{lemma}\label{lem:multilinear Strichartz}
	There is an absolute constant $B>0$ such that the following multilinear estimate holds.
	Let $I\subset \R$ and $s\in I$.
	For any $n,m\ge 0$, $V_1,\dots,V_n,V_1',\dots,V_m' \in \CL(I)$, and $\phi,\phi'\in L^2_x(\R)$,
	we have
	\begin{equation}\label{eq:mutlilinear Strichartz 2}
		\begin{aligned}
			&\No{S(t)W_{V_1,\dots,V_n}^{(n)}(t,s)S(-s)\phi \cdot \ov{S(t)W_{V'_1,\dots,V'_m}^{(m)}(t,s)S(-s) \phi'}}_{L^2_t(I;\dH^{1/2}_x(\R))} \\
			&\quad \le B\lg n \rg^{1/2} \lg m\rg^{1/2} A^{n+m} \prod_{j=1}^n \|V_j\|_{L^2_t(I;\dH^{-1/2}_x(\R))}
			\prod_{k=1}^m \|V_k'\|_{L^2_t(I;\dH^{-1/2}_x(\R))}
			\|\phi\|_{L^2_x(\R)} \|\phi'\|_{L^2_x(\R)},
		\end{aligned}
	\end{equation}
	where $A$ is the constant in Theorem \ref{th:multilinear 1} and $\lg x \rg := (1+|x|^2)^{1/2}$.
\end{lemma}

\begin{proof}
	Let $J:=I\cap [s,\I)$.
	We only consider the $L^2_t(J;\dH^{1/2}_x(\R))$ norm because we can estimate $L^2_t(I\setminus J;\dH^{1/2}_x(\R))$ in a similar way.
	By a density argument, we can assume $\phi,\phi'\in C_c^\I(\R)$.
	By a symmetry argument, it suffices to prove
	\begin{equation}\label{eq:T}
		\begin{aligned}
			&\No{S(t)W_{V_1,\dots,V_n}^{(n)}(t,s)S(-s)\phi \cdot \ov{S(t)W_{V'_1,\dots,V'_m}^{(m)}(t,s)S(-s) \phi'}}_{L^2_t(s,T;\dH^{1/2}_x(\R))} \\
			&\quad \le B (1+m) A^{n+m} \prod_{j=1}^n \|V_j\|_{L^2_t(J;\dH^{-1/2}_x(\R))}
			\prod_{k=1}^m \|V_k'\|_{L^2_t(J;\dH^{-1/2}_x(\R))}
			\|\phi\|_{L^2_x(\R)} \|\phi'\|_{L^2_x(\R)}
		\end{aligned}
	\end{equation}
	for all $T\in J$.
	Fix an arbitrary $T\in J$ and $h\in \CL(J)$.
	Then, we have
	\begin{align}
		&\int_s^T \int_\R h(t,x) S(t)W_{V_1,\dots,V_n}^{(n)}(t,s)S(-s)\phi \cdot \ov{S(t)W_{V'_1,\dots,V'_m}^{(m)}(t,s) S(-s)\phi'} dx dt \\
		&\quad = \int_\R \ov{S(-s)\phi'(x)} \K{\int_s^T W_{V'_1,\dots,V'_m}^{(m)}(t,s)^* S[h](t) W_{V_1,\dots,V_n}^{(n)}(t,s) dt } S(-s)\phi(x) dx \\
		&\quad \le \|\phi\|_{L^2_x}\|\phi'\|_{L^2_x} \No{\int_s^T W_{V'_1,\dots,V'_m}^{(m)}(t,s)^* S[h](t) W_{V_1,\dots,V_n}^{(n)}(t,s) dt}_{\CB}=:\|\phi\|_{L^2_x}\|\phi'\|_{L^2_x}\|\SA_m\|_{\CB} .
	\end{align}
	Since
	\begin{equation}
		S[Q_1](t)W_{Q_2,\dots,Q_k}^{(k-1)}(t,s) = i\pl_t W_{Q_1,\dots,Q_k}^{(k)}(t,s), 
	\end{equation}
	integration by parts implies
	\begin{align}
		\SA_m
		= i W_{\ov{V_m'},\dots,\ov{V_1'},h,V_1,\dots,V_n}^{(n+m+1)}(T,s)+ 
		i \sum_{k=0}^{m-1} \K{W_{V_{k+1}',\dots,V_m'}^{(m-k)}(T,s)}^* W_{\ov{V_k'},\dots,\ov{V_1'},h,V_1,\dots,V_n}^{(n+1+k)}(T,s).
	\end{align}
	Therefore, by Theorem \ref{th:multilinear 1}, we have
	\begin{align}
		\|\SA_m\|_{\CB}
		&\le \frac{\ep A^{n+m+1}}{(n+m+1)^2}
		\prod_{k=1}^m \|V_{k}'\|_{L^2_t(J;\dH^{-1/2}_x)} 
		\|h\|_{L^2_t(J;\dH^{-1/2}_x)} 
		\prod_{j=1}^n \|V_j\|_{L^2_t(J;\dH^{-1/2}_x)} \\
		&\qquad + \sum_{k=0}^{m-1} \No{W_{\ov{V_m'},\dots,\ov{V_{k+1}'}}^{(m-k)}(T)}_{\CB}
		\No{W_{\ov{V_k'},\dots,\ov{V_1'},h,V_1,\dots,V_n}^{(n+1+k)}(T)}_{\CB} \\
		&\le \ep A^{n+m+1}
		\prod_{k=1}^m \|V_{k}'\|_{L^2_t(J;\dH^{-1/2}_x)} 
		\|h\|_{L^2_t(J;\dH^{-1/2}_x)} 
		\prod_{j=1}^n \|V_j\|_{L^2_t(J;\dH^{-1/2}_x)} \\
		&\qquad + C \ep^2 m A^{n+m+1} 
		\prod_{k=1}^m \|V_{k}'\|_{L^2_t(J;\dH^{-1/2}_x)} 
		\|h\|_{L^2_t(J;\dH^{-1/2}_x)} 
		\prod_{j=1}^n \|V_j\|_{L^2_t(J;\dH^{-1/2}_x)} ,
	\end{align}
	where $C>0$ is an absolute constant.
	Taking $B =(\ep + C\ep^2) A$, we conclude that
	\begin{equation}
		\|\SA_m\|_{\CB}
		\le B(m+1)A^{n+m} \prod_{k=1}^m \|V_{k}'\|_{L^2_t(J;\dH^{-1/2}_x)} 
		\|h\|_{L^2_t(J;\dH^{-1/2}_x)} 
		\prod_{j=1}^n \|V_j\|_{L^2_t(J;\dH^{-1/2}_x)} ,
	\end{equation}
	and hence \eqref{eq:T}.
\end{proof}

We now obtain one of the key estimates in this paper.
\begin{lemma}[Ozawa--Tsutsumi type bilinear Strichartz estimate]\label{lem:key}
Let $I\subset \R$ be an interval.
	Assume that $V$ and $V'$ satisfy
	$$r:=\max\K{A\|V\|_{L^2_t(I;\dH^{-1/2}_x(\R))}, A\|V'\|_{L^2_t(I;\dH^{-1/2}_x(\R))}}\le \frac{1}{2},$$
	where $A$ is the constant in Theorem \ref{th:multilinear 1}.
	Then, for any $s,s'\in I$,
	\begin{align}
		&\No{S_V(t,s)\phi \cdot \ov{S_{V'}(t,s')\phi'}}_{L^2_t(I;\dH^{1/2}_x(\R))}
		\le C \|\phi\|_{L^2_x(\R)} \|\phi'\|_{L^2_x(\R)}, \label{eq:a priori}\\
		&\No{S_V(t,s)\phi \cdot \ov{S_V(t,s')\phi'} - S_{V'}(t,s)\phi \cdot \ov{S_{V'}(t,s')\phi'}}_{L^2_t(I;\dH^{1/2}_x(\R))} \\
		&\qquad \qquad \qquad \qquad \qquad \qquad \qquad \qquad
		\le C\|V-V'\|_{L^2_t (I;\dH^{-1/2}_x(\R))}\|\phi\|_{L^2_x} \|\phi'\|_{L^2_x(\R)},\label{eq:difference}
	\end{align}
	where $C>0$ is an absolute constant.
\end{lemma}

\begin{proof}
	First, we prove \eqref{eq:a priori}.
	By Lemma \ref{lem:multilinear Strichartz} and \eqref{eq:bound of UV}, we have 
	\begin{align}
		&\No{S_V(t,s)\phi\cdot \ov{S_{V'}(t,s')\phi'}}_{L^2_t(I;\dH^{1/2}_x)}
		= \No{S_V(t,s)\phi\cdot \ov{S_{V'}(t,s)S_{V'}(s,s')\phi'}}_{L^2_t(I;\dH^{1/2}_x)} \\ 
		&\quad \le \sum_{n,m\ge 0} \No{S(t)W_V^{(n)}(t,s)S(-s)\phi \cdot \ov{S(t)W_{V'}^{(m)}(t,s)S(-s)S_{V'}(s,s')\phi'}}_{L^2_t(I;\dH^{1/2}_x)} \\
		&\quad \le B \sum_{n,m\ge 0} \lg n \rg^{1/2} \lg m \rg^{1/2} r^{n+m} \|\phi\|_{L^2_x} \|S_{V'}(s,s')\phi'\|_{L^2_x}  \le 100B \|\phi\|_{L^2_x} \|\phi'\|_{L^2_x},
	\end{align}
	where $B$ is the constant in Lemma \ref{lem:multilinear Strichartz}.
		
	Next, we prove \eqref{eq:difference}.
	Note that
	\begin{align}
		\text{LHS of \eqref{eq:difference}}&\le \No{S_V(t,s)\phi \cdot \ov{(S_V(t,s')-S_{V'}(t,s'))\phi'}}_{L^2_t(I;\dH^{1/2}_x)} \\
		&\quad + \No{(S_V(t,s) - S_{V'}(t,s))\phi \cdot \ov{S_{V'}(t,s')\phi'}}_{L^2_t(I;\dH^{1/2}_x)}  =:\SA+\SB.
	\end{align}
	We only estimate $\SA$ because we can deal with $\SB$ in the same way.
	We have
	\begin{align}
		\SA
		&\le \No{S_V(t,s)\phi \cdot \ov{(S_V(t,s) - S_{V'}(t,s))S_V(s,s')\phi'}}_{L^2_t(I;\dH^{1/2}_x)} \\
		&\quad + \No{S_V(t,s)\phi \cdot \ov{S_{V'}(t,s)(S_V(s,s') - S_{V'}(s,s'))\phi'}}_{L^2_t(I;\dH^{1/2}_x)} =: \SA_1 + \SA_2.
	\end{align}
	By Lemma \ref{lem:multilinear Strichartz} and \eqref{eq:bound of UV}, we have
	\begin{align}
		\SA_1&\le \sum_{n\ge 0, m\ge 1} \Big\|S(t)W_V^{(n)}(t,s) S(-s)\phi \cdot\ov{S(t)(W_{V}^{(m)}(t,s) - W_{V'}^{(m)}(t,s)) S(-s)S_{V}(s,s')\phi'}\Big\|_{L^2_t(I;\dH^{1/2}_x)} \\
		&\le \sum_{n\ge 0, m\ge 1} \bigg(\No{S(t)W_V^{(n)}(t,s)S(-s)\phi\cdot\ov{S(t)W_{V-V',V,\dots,V}^{(m)}(t,s)S(-s)S_{V}(s,s')\phi'}}_{L^2_t(I;\dH^{1/2}_x)} + \cdots \\
		&\qquad\qquad \cdots  + \No{S(t)W_V^{(n)}(t,s)S(-s)\phi\cdot\ov{S(t)W_{V',\dots,V',V-V'}^{(m)}(t,s)S(-s)S_{V}(s,s')\phi'}}_{L^2_t(I;\dH^{1/2}_x)}\bigg) \\
		&\le B A \|V-V'\|_{L^2_t(I;\dH^{-1/2}_x)} \|\phi\|_{L^2_x} \|S_{V}(s,s')\phi'\|_{L^2_x} \sum_{n\ge 0, m\ge 1} m \lg n\rg^{1/2} \lg m\rg^{1/2} r^{n+m-1}  \\
		&\le 100BA \|V-V'\|_{L^2_t(I;\dH^{-1/2}_x)} \|\phi\|_{L^2_x} \|\phi'\|_{L^2_x}.
	\end{align}
	For $\SA_2$, by \eqref{eq:a priori} and Lemma \ref{lem:difference estimate}, we have
	\begin{align}
		\SA_2
		&\le 100B \|\phi\|_{L^2_x} \|(S_V(s,s') - S_{V'}(s,s'))\phi'\|_{L^2_x} \le 400BA\|V-V'\|_{L^2_t(I;\dH^{-1/2}_x)} \|\phi\|_{L^2_x} \|\phi'\|_{L^2_x}.
	\end{align}
	Collecting the above arguments, we have
	\begin{align}
		\SA \le 500BA\|V-V'\|_{L^2_t(I;\dH^{-1/2}_x)} \|\phi\|_{L^2_x} \|\phi'\|_{L^2_x}.
	\end{align}
	Similarly, we have
	\begin{align}
		\SB \le 500BA \|V-V'\|_{L^2_t(I;\dH^{-1/2}_x)} \|\phi\|_{L^2_x} \|\phi'\|_{L^2_x}.
	\end{align}
	Choosing $C:=1000BA$, we obtain \eqref{eq:difference}.
\end{proof}

\section{Well-posedness and scattering} \label{sec:CM}
In this section, we give a proof of Theorems \ref{th:LWP GWP} and \ref{th:CM scattering}.
Recall that the definition of the solution is given in Definition \ref{def:solution}.
The arguments in this section are based on the construction of $S_V$ and the Strichartz estimates established in Section \ref{sec:SV Stri}.
The arguments in this section are much more standard than those in Sections \ref{sec:multilinear estimate} and \ref{sec:SV Stri}.
	
\subsection{Local well-posedness}\label{subsec:LWP}
First, we prove the local well-posedness part of Theorem \ref{th:LWP GWP}.
Consider a general $p:\R_\xi \to \C$ and $q\in \C$.
Choose $R=R(\|p\|_{L^\I_{-1}})>0$ such that
\begin{equation}\label{eq:choice of R}
	A \|p\|_{L^\I_{-1}} R \le \tw.
\end{equation}
Set
$$E_{T,R} := \Ck{(u,\varrho): \|u\|_{\CX(I)} \le R, \quad \|\varrho\|_{L^2_t(I;\dH^{1/2}_x)} \le R },$$
where $T>0$ and $I=[-T,T]$ will be chosen later.
See \eqref{eq:definition of X} for the definition of the $\CX(I)$ norm.
Note that $E_{T,R}$ is a complete metric space with respect to 
$$d((u,\varrho),(u',\varrho')) := \|u-u'\|_{\CX(I)} + \|\varrho-\varrho'\|_{L^2_t(I;\dH^{1/2}_x)}.$$
Define $F[u,\varrho]=(F_1[u,\varrho],F_2[u,\varrho])$, where
\begin{align}
	&F_1[u,\varrho](t) := S_{p(D) \varrho}(t)\phi - iq\int_0^t S_{p(D)\varrho}(t,\ta)|u(\ta)|^2u(\ta)d\ta. \\
	&F_2[u,\varrho](t) := |F_1[u,\varrho](t)|^2.
\end{align}

As a first step of the proof, we prove that $F:E_{T,R} \to E_{T,R}$ is well-defined.
\begin{lemma}
	Choose both $\ep_0=\ep_0(\|p\|_{L^\I_{-1}},|q|)>0$ and $T=T(\|p\|_{L^\I_{-1}},|q|)>0$ sufficiently small.
	Assume that $\|\phi\|_{L^2_x(\R)} \le \ep_0$.
	Then, for any $(u,\varrho)\in E_{T,R}$, we have $F[u,\varrho]\in E_{T,R}$.
\end{lemma}

\begin{proof}
Since $A\|p(D)\varrho\|_{L^2_t \dH^{-1/2}_x}\le 1/2$ by the choice of $R>0$, Lemma \ref{lem:standard Stri} implies
\begin{align}
	\|F_1[u,\varrho]\|_{\CX(I)}
	&\le C_{\textup{Stri}}\|\phi\|_{L^2_x} + |q| C_{\textup{Stri}}\||u|^2 u\|_{L^1_t(I;L^2_x)} \\
	&\le C_{\textup{Stri}} \ep_0 + \sqrt{2}|q|C_{\textup{Stri}}  R^3T^{1/2}.
\end{align}
Choosing sufficiently small $\ep_0=\ep_0(\|p\|_{L^\I_{-1}},|q|)>0$ and $T=T(\|p\|_{L^\I_{-1}},|q|)>0$, we obtain
\begin{align}
	\|F_1[u,\varrho]\|_{\CX(I)} \le R.
\end{align}

Next, we consider $F_2$. We have
\begin{multline}
	\|F_2[u,\varrho]\|_{L^2_t(I;\dH^{1/2}_x)}
	\le  \No{|S_{p(D) \varrho}(t)\phi|^2}_{L^2_t(I;\dH^{1/2}_x)}
	      + |q|^2 \No{\abs{\int_0^t S_{p(D)\varrho}(t,\ta)|u(\ta)|^2u(\ta)d\ta}^2}_{L^2_t(I;\dH^{1/2}_x)} \\
	\quad + 2|q| \No{\ov{S_{p(D)\varrho}(t)\phi} \int_0^t S_{p(D)\varrho}(t,\ta)|u(\ta)|^2u(\ta)d\ta}_{L^2_t(I;\dH^{1/2}_x)} =: \SA + \SB + \SC.
\end{multline}
By Lemma \ref{lem:key},
\begin{equation}
	\SA \le C \|\phi\|_{L^2_x}^2 \le C\ep_0^2.
\end{equation}
Again by Lemma \ref{lem:key}, we have
\begin{align}
	\SB
	&\le |q|^2  \int_I \int_I \No{S_{p(D)\varrho}(t,\ta)|u(\ta)|^2u(\ta) \cdot \ov{S_{p(D)\varrho}(t,\ta')|u(\ta')|^2u(\ta')}}_{L^2_t \dH^{1/2}_x} d\ta d\ta' \\
	&\le  C |q|^2 \int_I \int_I \||u(\ta)|^2u(\ta)\|_{L^2_x} \||u(\ta')|^2u(\ta')\|_{L^2_x} d\ta d\ta'
	\le 4C |q|^2 R^6 T. 
\end{align}
In the same way, we conclude that
\begin{align}
	\SC
	\le 2C|q|R^3 \ep_0 T^{1/2}.
\end{align}
Hence, collecting the estimates for $\SA$, $\SB$, and $\SC$, we conclude that
\begin{align}
	\|F_2[u,\varrho]\|_{L^2_t(I;\dH^{1/2}_x)}
	&\le C \ep_0^2 + 4C |q|^2 R^6 T + 2C|q|R^3 \ep_0 T^{1/2}.
\end{align}
Choosing sufficiently small $\ep_0=\ep_0(\|p\|_{L^\I_{-1}},|q|)>0$ and $T=T(\|p\|_{L^\I_{-1}},|q|)>0$, we obtain
\begin{align}
	\|F_2[u,\varrho]\|_{L^2_t(I;\dH^{1/2}_x)} \le R,
\end{align}
which completes the proof.
\end{proof}

To complete the existence proof, it suffices to prove
\begin{lemma}\label{lem:F1 and F2 difference}
If we choose sufficiently small $\ep_0 = \ep_0(\|p\|_{L^\I_{-1}},|q|)$ and $T=T(\|p\|_{L^\I_{-1}},|q|)>0$, then for any $(u,\varrho), (u',\varrho') \in E_{T,R}$, it holds that
\begin{align}
		&\No{F_1[u,\varrho] - F_1[u',\varrho']}_{\CX(I)} \le \ft \K{\|u-u'\|_{\CX(I)} + \|\varrho-\varrho'\|_{L^2_t(I;\dH^{1/2}_x)}}, \label{eq:F1 difference}\\
		&\No{F_2[u,\varrho] - F_2[u',\varrho']}_{L^2_t(I;\dH^{1/2}_x(\R))} \le \ft \K{\|u-u'\|_{\CX(I)} + \|\varrho-\varrho'\|_{L^2_t(I;\dH^{1/2}_x)}} \label{eq:F2 difference}.
\end{align}
\end{lemma}
If Lemma \ref{lem:F1 and F2 difference} is true, $F:E_{T,R} \to E_{T,R}$ is a contraction mapping.
Hence, we obtain a unique fixed point in $E_{T,R}$.
The Lipschitz continuity also follows from the standard argument.
Finally, we give a proof of Lemma \ref{lem:F1 and F2 difference}.
\begin{proof}[Proof of Lemma \ref{lem:F1 and F2 difference}]
	\noindent \textbf{Step 1: Proof of \eqref{eq:F1 difference}.}
	First, we have
	\begin{align}
		&\|F_1[u,\varrho] - F_1[u',\varrho']\|_{\CX(I)}
		\le \No{S_{p(D)\varrho}(t)\phi - S_{p(D)\varrho'}(t)\phi}_{\CX(I)} \\
		&\qquad \qquad \qquad \qquad + |q|\No{\int_0^t S_{p(D)\varrho}(t,\ta)(|u|^2 u)(\ta)d\ta - \int_0^t S_{p(D)\varrho'}(t,\ta)(|u|^2 u)(\ta)d\ta }_{\CX(I)} \\
		&\qquad \qquad \qquad \qquad + |q|\No{\int_0^t S_{p(D)\varrho'}(t,\ta)(|u|^2 u - |u'|^2 u')(\ta)d\ta }_{\CX(I)} = \SA + \SB +\SC.
	\end{align}
	By Lemma \ref{lem:difference SV},
	\begin{align}
		\SA \le C_{\textup{Stri}}' \|p(D)(\varrho-\varrho')\|_{L^2_t(I;\dH^{-1/2}_x)} \|\phi\|_{L^2_x} 
		\le C_{\textup{Stri}}'  \|p\|_{L^\I_{-1}} \ep_0 \|\varrho-\varrho'\|_{L^2_t(I;\dH^{1/2}_x)}.
	\end{align}
	Again by Lemma \ref{lem:difference SV}, 
	\begin{align}
		\SB
		&\le C_{\textup{Stri}}' |q| \|p(D)(\varrho-\varrho')\|_{L^2_t(I;\dH^{-1/2}_x)} \||u|^2 u\|_{L^1_t(I;L^2_x)} \\
		&\le \sqrt{2}C_{\textup{Stri}}' |q| \|p\|_{L^\I_{-1}} R^3  T^{1/2}\|\varrho-\varrho'\|_{L^2_t(I;\dH^{1/2}_x)}.
	\end{align}
	By Lemma \ref{lem:standard Stri}, 
	\begin{align}
		\SC \le C_{\textup{Stri}} |q| \||u|^2 u - |u'|^2 u'\|_{L^1_t(I;L^2_x)}
		\le 3\sqrt{2} C_{\textup{Stri}} |q| R^2  T^{1/2}\|u-u'\|_{\CX(I)}.
	\end{align}
	Collecting the above arguments, we conclude that
	\begin{align}
		\|F_1[u,\varrho] - F_1[u',\varrho']\|_{\CX(I)}
		\le \ft \K{\|\varrho-\varrho'\|_{L^2_t(I;\dH^{1/2}_x)} + \|u-u'\|_{\CX(I)}}
	\end{align}
	by choosing sufficiently small $\ep_0=\ep_0(\|p\|_{L^\I_{-1}},|q|)>0$ and $T=T(\|p\|_{L^\I_{-1}},|q|)>0$.
	
	\noindent \textbf{Step 2: Proof of \eqref{eq:F2 difference}.}
	We have
	\begin{align}
		\|F_2[u,\varrho] - F_2[u',\varrho']\|_{L^2_t(I;\dH^{1/2}_x)}
		&\le \|F_2[u,\varrho] - F_2[u,\varrho']\|_{L^2_t(I;\dH^{1/2}_x)}
		+ \|F_2[u,\varrho'] - F_2[u',\varrho']\|_{L^2_t(I;\dH^{1/2}_x)} \\
		&=:\SD+\SE.
	\end{align}
	\noindent \textbf{Step 2.1: Estimate of $\SD$.}
	We have
	\begin{align}
		\SD &\le \No{|S_{p(D)\varrho}(t)\phi|^2 - |S_{p(D)\varrho'}(t)\phi|^2 }_{L^2_t(I;\dH^{1/2}_x)}  \\
			&\quad   + |q|^2 \No{\abs{\int_0^t S_{p(D)\varrho}(t,\ta)(|u|^2u)(\ta)d\ta}^2 - \abs{\int_0^t S_{p(D)\varrho'}(t,\ta)(|u|^2 u)(\ta) d\ta}^2}_{L^2_t(I;\dH^{1/2}_x)} \\
			&\quad  + 2|q| \bigg\|S_{p(D)\varrho}(t)\phi \int_0^t \ov{S_{p(D)\varrho}(t,\ta)(|u|^2u)(\ta)} d\ta \\
			&\qquad \qquad \qquad -S_{p(D)\varrho'}(t)\phi \int_0^t \ov{S_{p(D)\varrho'}(t,\ta)(|u|^2u)(\ta)} d\ta
		\bigg\|_{L^2_t(I;\dH^{1/2}_x)}=:\SD_1 + \SD_2 + \SD_3.
	\end{align}
	By Lemma \ref{lem:key},
	\begin{align}
		\SD_1 
		&\le C \|p(D)\varrho-p(D)\varrho'\|_{L^2_t(I;\dH^{-1/2}_x)} \|\phi\|_{L^2_x}^2 \\
		&\le C \|p\|_{L^\I_{-1}} \ep_0^2 \|\varrho-\varrho'\|_{L^2_t(I;\dH^{1/2}_x)}.
	\end{align}
	Again by Lemma \ref{lem:key}, 
	\begin{align}
		\SD_2 &\le |q|^2\int_I \int_I \Big\| S_{p(D)\varrho}(t,\ta)(|u|^2 u)(\ta) \cdot \ov{S_{p(D)\varrho}(t,\ta')(|u|^2u)(\ta')}  \\
		&\qquad \qquad \qquad \qquad \qquad  -S_{p(D)\varrho'}(t,\ta)(|u|^2 u)(\ta) \cdot \ov{S_{p(D)\varrho'}(t,\ta')(|u|^2u)(\ta')} \Big\|_{L^2_t(I;\dH^{1/2}_x)} d\ta d\ta'
		\\
		&\le C |q|^2 \|p(D)\varrho-p(D)\varrho'\|_{L^2_t(I;\dH^{-1/2}_x)} \K{\int_I \||u(\ta)|^2u(\ta)\|_{L^2_x} d\ta}^2\\
		&\le 2C |q|^2 \|p\|_{L^\I_{-1}}  R^6 T \|\varrho-\varrho'\|_{L^2_t(I;\dH^{1/2}_x)}.
	\end{align}
	Similarly, 
	\begin{align}
		\SD_3 &\le \sqrt{2}C |q| \|p\|_{L^\I_{-1}}  R^3 \ep_0T^{1/2} \|\varrho-\varrho'\|_{L^2_t(I;\dH^{1/2}_x)}.
	\end{align}
	Collecting the above estimates, we obtain
	\begin{equation}
		\SD \le \ft \|\varrho-\varrho'\|_{L^2_t(I;\dH^{1/2}_x)}
	\end{equation}
	by choosing sufficiently small $\ep_0=\ep_0(\|p\|_{L^\I_{-1}},|q|)>0$ and $T=T(\|p\|_{L^\I_{-1}},|q|)>0$.

	\noindent \textbf{Step 2.2: Estimate of $\SE$.}
	We have
	\begin{align}
		\SE &\le 2|q| \No{S_{p(D)\varrho'}(t)\phi \int_0^t \ov{S_{p(D)\varrho'}(t,\ta)(|u|^2u - |u'|^2 u')(\ta)} d\ta}_{L^2_t(I;\dH^{1/2}_x)}\\
		&\quad + |q|^2 \No{\abs{\int_0^t S_{p(D)\varrho'}(t,\ta)(|u|^2u)(\ta)d\ta}^2 - \abs{\int_0^t S_{p(D)\varrho'}(t,\ta)(|u'|^2 u')(\ta) d\ta}^2}_{L^2_t(I;\dH^{1/2}_x)} \\
		&=: \SE_1 +\SE_2.
	\end{align}
	By Lemma \ref{lem:key}, we have
	\begin{align}
		\SE_1 
		&\le 2|q| \int_I \No{S_{p(D)\varrho'}(t)\phi \ov{S_{p(D)\varrho'}(t,\ta)(|u|^2u - |u'|^2 u')(\ta)}}_{L^2_t(I;\dH^{1/2}_x)} d\ta \\
		&\le 2C|q|\ep_0 \int_I \No{|u(\ta)|^2u(\ta) - |u'(\ta)|^2 u'(\ta)}_{L^2_x} d\ta \\
		&\le 6C |q|R^2 \ep_0 T^{1/2} \|u-u'\|_{\CX(I)}.
	\end{align}
	Similarly, by Lemma \ref{lem:key}, we have
	\begin{align}
		\SE_2 &\le |q|^2 \No{\int_0^t S_{p(D)\varrho'}(t,\ta)(|u|^2u - |u'|^2u')(\ta)d\ta \int_0^t \ov{S_{p(D)\varrho'}(t,\ta)(|u|^2u)(\ta)} d\ta}_{L^2_t(I;\dH^{1/2}_x)} \\
		&\qquad + |q|^2 \No{\int_0^t S_{p(D)\varrho'}(t,\ta)(|u'|^2u')(\ta)d\ta \int_0^t \ov{S_{p(D)\varrho'}(t,\ta)(|u|^2u-|u'|^2 u')(\ta)} d\ta}_{L^2_t(I;\dH^{1/2}_x)}   \\
		&\le 12C  |q|^2 T R^5 \|u-u'\|_{\CX(I)}.
	\end{align}
	Collecting the above estimates, we obtain
	\begin{equation}
		\SE \le \ft \|u-u'\|_{\CX(I)}
	\end{equation}
	by choosing sufficiently small $\ep_0=\ep_0(\|p\|_{L^\I_{-1}},|q|)>0$ and $T=T(\|p\|_{L^\I_{-1}},|q|)>0$.	
	
	\noindent\textbf{Step 2.3: Conclusion.}
	Collecting the above estimates, we have
	\begin{align}
		\|F_2[u,\varrho] - F_2[u',\varrho']\|_{L^2_t(I;\dH^{1/2}_x)}
		&\le \SD+\SE \le \ft\K{ \|u-u'\|_{\CX(I)} + \|\varrho-\varrho'\|_{L^2_t(I;\dH^{1/2}_x)} }.
	\end{align}
\end{proof}

\subsection{Global well-posedness}\label{subsec:GWP}
First, we prove local-in-time mass conservation.
\begin{lemma}\label{lem:local in time conservation}
	Let $\ep_0$, $R$, and $T$ be small constants in Section \ref{subsec:LWP}, which depend only on $\|p\|_{L^\I_{-1}}$ and $|q|$.
	We write $I:=[-T,T]$.
	Let $\phi\in L^2_x(\R)$ satisfy $\|\phi\|_{L^2_x}\le \ep_0$ and $u\in C_t(I;L^2_x) \cap L^{4}_t(I;L^\I_x)$ be the solution to \eqref{eq:CMDNLS} such that $u(0)=\phi$, which was constructed in Section \ref{subsec:LWP}.
	Finally, assume that $p(\xi)$ is real-valued and $q\in \R$.
	Then, the mass is conserved, namely, we have
	$$\|u(t)\|_{L^2_x(\R)} = \|\phi\|_{L^2_x(\R)}$$
	for any $t\in I$.
\end{lemma}
\begin{proof}[Proof of Lemma \ref{lem:local in time conservation}]
Let $V := p(D)\varrho  \in L^2_t(I;\dH^{-1/2}_x)$.
Then, we can choose $V_j\in \CL(I)$ such that $V_j \to V$ in $L^2_t(I;\dH^{-1/2}_x)$.
Since $A\|V\|_{L^2_t(I;\dH^{-1/2}_x)} \le 1/2$, we can assume 
$$\sup_j A\|V_j\|_{L^2_t(I;\dH^{-1/2}_x)} \le \tw,$$
where $A$ is the constant in Theorem \ref{th:multilinear 1}.
Consider the equation for $u_j$:
\begin{align}
	u_j(t) := S_{V_j}(t)\phi - iq \int_0^t S_{V_j}(t,\ta)|u(\ta)|^2 u_j(\ta)d\ta.
\end{align}
By essentially the same arguments used in Section \ref{subsec:LWP}, we obtain a solution $u_j\in C_t(I;L^2_x)$ such that
\begin{equation}
	\sup_j \|u_j\|_{\CX(I)} \le R \quad \text{and} \quad \sup_{j}\||u_j|^2\|_{L^2_t(I;\dH^{1/2}_x)} \le R,
\end{equation}
where $R>0$ is chosen in \eqref{eq:choice of R}.
Note that $(u_j)_j$ approximates $u$. More precisely, we have
\begin{align}
	&\lim_{j\to \I}\|u-u_j\|_{C_t(I_0;L^2_x)} = 0, \label{eq:L^2 norm convergence} \\
	&\lim_{j\to \I}\||u|^2-|u_j|^2\|_{L^2_t(I_0;\dH^{1/2}_x)} = 0, \label{eq:H norm convergence}
\end{align}
where $I_0=[-T_0,T_0]$ and $T_0=T_0(\|p\|_{L^\I_{-1}}, |q|)\in (0,T]$ is chosen sufficiently small.
We omit the proof of \eqref{eq:L^2 norm convergence} and \eqref{eq:H norm convergence} because they can be proved in exactly the same way as in the proof of Lemma \ref{lem:F1 and F2 difference}.
	Since $V_j$ is sufficiently regular,
	\begin{equation}\label{eq:identity}
		\|u_j(t)\|_{L^2_x}^2 = \|\phi\|_{L^2_x}^2 + 2\Im\Dk{\int_0^t \int_{\R} (|\pl_x|^{-1/2}V_j(\ta,x)) (|\pl_x|^{1/2} |u_j(\ta,x)|^2)dxd\ta}.
	\end{equation}
	Passing to the limit using \eqref{eq:L^2 norm convergence} and \eqref{eq:H norm convergence}, we obtain
	\begin{equation}
		\|u(t)\|_{L^2_x}^2 = \|\phi\|_{L^2_x}^2 + 2\Im\Dk{\int_0^t \int_{\R} (|\pl_x|^{-1/2}V(\ta,x)) (|\pl_x|^{1/2} |u(\ta,x)|^2)dxd\ta}=\|\phi\|_{L^2_x}^2
	\end{equation}
	for all $t\in I_0$, where we used
	\begin{equation}
		\int_{\R} (|\pl_x|^{-1/2}V(\ta,x)) (|\pl_x|^{1/2} |u(\ta,x)|^2)dx
		= \int_\R p(\xi)|\wh{\varrho}(\ta,\xi)|^2 d\xi \in \R. 
	\end{equation}
	Now, let us regard $t=T_0$ as the initial time.
	Then, we can repeat the above argument to obtain $\|u(t)\|_{L^2_x}=\|\phi\|_{L^2_x}$ for all $-T_0 \le t \le \min(2T_0,T)$.
	A similar backward-in-time argument proves $\|u(t)\|_{L^2_x}=\|\phi\|_{L^2_x}$ for all $-\min(2T_0,T)\le t \le \min(2T_0,T)$.
	Repeating the same argument, we obtain mass conservation on $I$. 	
\end{proof}

Finally, we give a proof of Theorem \ref{th:LWP GWP} $(iii)$.
By Lemma \ref{lem:local in time conservation}, we have $\|u(t)\|_{L^2_x} = \|\phi\|_{L^2_x}$ for all $t\in[-T,T]$.
Now, let us regard $t=T$ as the initial time.
Then, we can repeat the proofs in Section \ref{subsec:LWP} to construct an extended solution $u\in C_t([-T,2T];L^2_x)$.
A similar backward-in-time argument also yields an extended solution $u\in C_t([-2T,2T];L^2_x)$.
Furthermore, by exactly the same argument as in the proof of Lemma \ref{lem:local in time conservation}, we obtain $\|u(t)\|_{L^2_x} = \|\phi\|_{L^2_x}$ for all $t\in [-2T,2T]$.
Repeating the above argument, we obtain the global solution $u\in C_t(\R;L^2_x)$ such that $\|u(t)\|_{L^2_x} = \|\phi\|_{L^2_x}$ for all $t\in\R$.

\subsection{Improved uniqueness}
As the final step of the proof of Theorem \ref{th:LWP GWP}, we prove $(iv)$.
We can assume $I'=[0,T'_+)$ because we can prove uniqueness on $(-T'_-,0]$ in the same way.
Let $u' \in C_t(I';L^2_x)\cap L^4_{\loc,t}(I';L^\I_x)$ be a solution to \eqref{eq:CMDNLS} with the initial condition $u'(0)=\phi$, where $\|\phi\|_{L^2_x} \le \ep_0$.
Moreover, we assume $\varrho' := |u'|^2 \in L^2_{\loc,t}(I';\dH^{1/2}_x)$.
Fix an arbitrary $0<T'_0 < T_+'$.
Then, there exists a partition of $[0,T_0']$, $0=r_0 < \cdots < r_N = T_0'$ such that $|r_{j+1}-r_j|\le T$ and
\begin{align}
	\|u'\|_{L^4_t(r_j,r_{j+1};L^\I_x)} \le \frac{R}{2},\quad \|\varrho'\|_{L^2_t(r_j,r_{j+1};\dH^{1/2}_x)} \le R
\end{align}
for all $j=0,1,\dots,N-1$, where $T,R>0$ are constants chosen in Section \ref{subsec:LWP}.
Moreover, by Lemma \ref{lem:standard Stri} and the smallness of $\ep_0$, we get
$$\|u'\|_{C_t(r_0,r_1;L^2_x)}\le \frac{R}{2}.$$
Hence, the same argument as in the proof of Lemma \ref{lem:F1 and F2 difference} implies $u(t)=u'(t)$ for all $t\in [0,r_1]$.
Hence, in particular, we have $\|u'(t)\|_{L^2_x} = \|\phi\|_{L^2_x}$ for all $t\in[0,r_1]$.
Now, let us regard $t=r_1$ as the initial time.
Then, by Lemma \ref{lem:standard Stri} and the smallness of $\ep_0$, we get
$$\|u'\|_{C_t(r_1,r_2;L^2_x)}\le \frac{R}{2}.$$
Hence, we can repeat the same argument as above to obtain $u(t)=u'(t)$ for all $t\in [0,r_2]$.
Repeating the same argument, we obtain $u(t)=u'(t)$ for all $t\in [0,T_0']$.
Since $T_0'$ was arbitrary, we obtain $u(t)=u'(t)$ for all $t\in [0,T'_+)$.

\subsection{Scattering}
Finally, we give a proof of Theorem \ref{th:CM scattering}.
Let $R>0$ be such that $A  (1+\|p\|_{L^\I_{-1}}) R\le 1/2$.
Let 
\begin{equation}
	E_R := \Ck{\varrho \in L^2_t(\R;\dH^{1/2}_x) : \|\varrho\|_{L^2_t \dH^{1/2}_x} \le R}.
\end{equation}
Define $F_3$ by
\begin{equation}
	F_3[\varrho](t):= |S_{p(D)\varrho}(t)\phi|^2.
\end{equation}
Let $\varrho \in E_R$. Then,
$$A\|p(D)\varrho\|_{L^2_t \dH^{-1/2}_x} \le A \|p\|_{L^\I_{-1}} R \le \tw.$$
Thus, Lemma \ref{lem:key} implies
\begin{equation}
	\|F_3[\varrho]\|_{L^2_t \dH^{1/2}_x}
	\le C \|\phi\|_{L^2_x}^2 \le C \ep_0^2.
\end{equation}
Therefore, if we choose sufficiently small $\ep_0 = \ep_0(\|p\|_{L^\I_{-1}})$, we find $F_3[\varrho]\in E_R$.
Next, let $\varrho_1,\varrho_2\in E_R$.
Again, by Lemma \ref{lem:key}, we have
	\begin{align}
	\No{F_3[\varrho_1] - F_3[\varrho_2]}_{L^2_t \dH^{1/2}_x}
	&\le C \|p(D)\varrho_1 - p(D)\varrho_2\|_{L^2_t \dH^{-1/2}_x} \ep_0^2 \\
	&\le C \|p\|_{L^\I_{-1}}  \ep_0^2\|\varrho_1-\varrho_2\|_{L^2_t \dH^{1/2}_x}.
	\end{align}
Therefore, if we choose sufficiently small $\ep_0 = \ep_0(\|p\|_{L^\I_{-1}} )$, we obtain
	\begin{align}
	\No{F_3[\varrho_1] - F_3[\varrho_2]}_{L^2_t \dH^{1/2}_x}
	&\le \tw \|\varrho_1-\varrho_2\|_{L^2_t \dH^{1/2}_x}.
\end{align}
	Hence, $F_3:E_R\to E_R$ is a contraction mapping, and the fixed point theorem ensures that there exists a unique $\varrho \in E_R$ such that
	\begin{equation}
		\varrho = |S_{p(D)\varrho}(t)\phi|^2.
	\end{equation}
From the above argument, the solution to \eqref{eq:CMDNLS} is given by
\begin{equation}
	u(t) = S_{p(D)\varrho}(t) \phi.
\end{equation}

Let $u,u'\in C(\R;L^2_x)$ be two solutions with the same initial data $\phi$ such that
\begin{equation}
	\max\K{\||u|^2\|_{L^2_t(\R;\dH^{1/2}_x)}, \||u'|^2\|_{L^2_t(\R;\dH^{1/2}_x)}} \le R.
\end{equation}
Then, we have $\varrho:=|u|^2 \in E_R$ and $\varrho'=|u'|^2\in E_R$.
Moreover, both $\varrho$ and $\varrho'$ are fixed points of $F_3$.
Hence, we have $\varrho=\varrho'$, which implies $u=u'$.

Finally, since $V := p(D) \varrho \in L^2_t(\R;\dH^{-1/2}_x)$, Theorem \ref{th:construction of SV} $(v)$ implies that $u(t)$ scatters as $t\to\pm\I$.
More precisely,
\begin{equation}
	\lim_{t\to\pm\I}\No{S(-t)u(t) - \phi_\pm }_{L^2_x} = 0, \quad \phi_\pm := W_V^\pm \phi.
\end{equation}

\end{document}